\documentclass{amsart}

\usepackage{graphicx}

\usepackage{xcolor}

\usepackage{amsmath, amssymb, amsfonts, amsthm, fullpage}
\usepackage{graphics,graphicx}
\usepackage{accents}
\usepackage{color}

\usepackage{algorithm}
\usepackage{algorithmicx}
\usepackage{algpseudocode}

\usepackage[plainpages=false,pdfpagelabels,hidelinks,unicode]{hyperref}

\newcommand{\norma}[1]{{\left\vert\kern-0.25ex\left\vert\kern-0.25ex\left\vert #1
    \right\vert\kern-0.25ex\right\vert\kern-0.25ex\right\vert}}
\newcommand{\sech}{\text{sech}}

\newcommand{\bd}{\mathbf{d}}
\newcommand{\by}{\mathbf{y}}
\newcommand{\teta}{\tilde{\eta}}
\newcommand{\tu}{\tilde{u}}
\newcommand{\tw}{\tilde{w}}
\newcommand{\tv}{\tilde{v}}
\newcommand{\tf}{\tilde{f}}
\newcommand{\tg}{\tilde{g}}

\newcommand{\Alpha}{\mathrm{A}}
\newcommand{\Beta}{\mathrm{B}}

\newcommand{\bw}{\mathbf{w}}

\newcommand{\bF}{\mathbf{f}}

\newcommand{\bS}{\mathbf{S}}

\newcommand{\bchi}{\boldsymbol{\chi}}

\newcommand{\tbn}[1]{{\left\vert\kern-0.25ex\left\vert\kern-0.25ex\left\vert #1 \right\vert\kern-0.25ex\right\vert\kern-0.25ex\right\vert}}

\newcommand{\dd}{\,\mathrm{d}}

\newtheorem{remark}{Remark}[section]

\newtheorem{proposition}{Proposition}[section]

\AtBeginDocument{%
  \setlength{\belowdisplayskip}{5pt} \setlength{\belowdisplayshortskip}{5pt}
  \setlength{\abovedisplayskip}{5pt} \setlength{\abovedisplayshortskip}{5pt}
}

\begin{document}

\title[A conservative fully-discrete numerical method for the regularised shallow water wave equations]{A conservative fully-discrete numerical method for the regularised shallow water wave equations}

\author{Dimitrios Mitsotakis}
\address{\textbf{D.~Mitsotakis:} Victoria University of Wellington, School of Mathematics and Statistics, PO Box 600, Wellington 6140, New Zealand}
\email{dimitrios.mitsotakis@vuw.ac.nz}

\author{Hendrik Ranocha}
\address{\textbf{H.~Ranocha:} King Abdullah University of Science and Technology (KAUST), Computer Electrical and Mathematical Science and Engineering Division (CEMSE), Thuwal, 23955-6900, Saudi Arabia }
\email{mail@ranocha.de}

\author{David I. Ketcheson}
\address{\textbf{D.~I.~Ketcheson:} King Abdullah University of Science and Technology (KAUST), Computer Electrical and Mathematical Science and Engineering Division (CEMSE), Thuwal, 23955-6900, Saudi Arabia }
\email{david.ketcheson@kaust.edu.sa}

\author{Endre S\"{u}li}
\address{\textbf{E.~S\"{u}li:} University   of   Oxford, Mathematical   Institute,   Oxford   OX2   6GG,   UK }
\email{Endre.Suli@maths.ox.ac.uk}


\thanks{E. S\"uli is grateful to the School of Mathematics and Statistics of the Victoria University of Wellington for the kind hospitality during his visit as the year 2015 London Mathematical Society/New Zealand Mathematical Society Forder Lecturer. Research reported in this publication was supported by the
King Abdullah University of Science and Technology (KAUST)}


\date{\today}



\begin{abstract}
The paper proposes a new, conservative fully-discrete scheme for the numerical
solution of the regularised shallow water Boussinesq system of equations in the
cases of periodic and reflective boundary conditions. The particular system is one
of a class of equations derived recently and can be used in practical simulations
to describe the propagation of weakly nonlinear and weakly dispersive long water
waves, such as tsunamis. Studies of small-amplitude long waves usually require
long-time simulations in order to investigate scenarios such as the overtaking
collision of two solitary waves or the propagation of transoceanic tsunamis.
For long-time simulations of non-dissipative waves such as solitary waves, the
preservation of the total energy by the numerical method can be crucial in the
quality of the approximation. The new conservative fully-discrete method consists
of a Galerkin finite element method for spatial semidiscretisation and an
explicit relaxation Runge--Kutta scheme for integration in time. The Galerkin
method is expressed and implemented in the framework of mixed finite element
methods. The paper provides an extended experimental study of the accuracy and
convergence properties of the new numerical method. The experiments reveal a
new convergence pattern compared to standard Galerkin methods.
\end{abstract}

\maketitle

\section{Introduction}

In this paper we are concerned with the numerical solution of the
regularised shallow water system of Boussinesq equations, also known as the BBM-BBM
system.  The regularised shallow water equations
derived in \cite{BC1998,BCS2002} from the incompressible Euler equations for
free-surface flows may be written in dimensional  variables as
\begin{align}
\begin{aligned}\label{eq:bbmbbm}
\eta_t+[(D+\eta) u]_x-\frac{1}{6} D^2 \eta_{xxt}&=0,\\
u_t+g\eta_x+uu_x-\frac{1}{6}D^2 u_{xxt}&=0,
\end{aligned}
\end{align}
where $g$ is acceleration due to gravity, $D$ is the depth of the impenetrable horizontal sea floor, $\eta=\eta(x,t)$,
$u=u(x,t)$ are real functions defined at the horizontal location $x$ and time
$t\geq 0$ denoting the free-surface elevation of the water and the horizontal
velocity at water level $\sqrt{2/3}D$. The regularised shallow water system
describes the two-way propagation of long crested waves of small amplitude, and
particularly, respecting the physics of water waves, solutions of system
(\ref{eq:bbmbbm}) are weakly nonlinear and weakly dispersive. In that sense,
system (\ref{eq:bbmbbm}) generalises the non-dispersive shallow water system
\begin{align}\label{eq:shallowwater}
\begin{aligned}
\eta_t+[(D+\eta)u]_x&=0,\\
u_t + g\eta_x+uu_x&=0,
\end{aligned}
\end{align}
in the context of nonlinear and dispersive water waves. Both systems
(\ref{eq:bbmbbm}) and (\ref{eq:shallowwater}) conserve the same approximation
of the total energy (kinetic plus potential) in the interval $I=(-\infty,\infty)$, defined by
\begin{align}\label{eq:energy}
\mathcal{E} (t; \eta,u):= \frac{1}{2} \int_I g\eta^2+(D+\eta)u^2 \dd x,
\end{align}
in the sense that their solutions satisfy
\begin{align}\label{eq:energycons}
\frac{\dd}{\dd t} \mathcal{E}(t;\eta,u)=0.
\end{align}
In addition to the energy conservation (\ref{eq:energycons}) the solutions of
the regularised shallow water equations (\ref{eq:bbmbbm}) conserve the following
quantities

\begin{align}
\mathcal{M}(t;\eta)&:=\int_I \eta \dd x, \quad \text{(mass)} \label{eq:mass1},\\
\mathcal{I}(t; u)&:=\int_I u \dd x,  \quad \text{(momentum)} \label{eq:mass2},\\
{\mathcal H}(t; \eta,u)&:=\int_I \eta u+\frac{1}{6}D^2\eta_x u_x \dd x, \quad \text{(impulse)}.  \label{eq:h1}
\end{align}
The quantity $\mathcal{M}$ represents the total mass and $\mathcal{I}$ the momentum, while $\mathcal{H}$ can be thought of as a generalised energy that is usually referred to as the impulse functional \cite{Benj1984}.  There is a Hamiltonian structure that relates the quantities $\mathcal{H}$ and $\mathcal{E}$; namely
$\partial_t \nabla \mathcal{H}(t;\eta,u)+\partial_x\nabla \mathcal{E}(t;\eta,u)=0$. Equivalently, the system (\ref{eq:bbmbbm}) can be written in Hamiltonian form as
$$\partial_t(\eta,u)^T=J~\nabla \mathcal{E}(t;\eta,u),$$
where $$J=\begin{pmatrix}
0 & (I-D^2/6\partial_x^2)^{-1}\partial_x\\
(I-D^2/6\partial_x^2)^{-1}\partial_x & 0
\end{pmatrix},$$
with $\mathcal{E}$ playing the role of the Hamiltonian.
For more information about the conservation properties and their theoretical implications we refer to \cite{BCS2004}.

Existence and uniqueness of weak and classical solutions for the Cauchy problem
of system (\ref{eq:bbmbbm}) have been established only locally in time in
\cite{BC1998,BCS2002,BCS2004}, while global well-posedness
in time was proved in \cite{AABCW2006} under the condition that there is an $\alpha>0$ such that
the solution satisfies
$$D+\eta(x,t)\geq \alpha \quad \mbox{ for all }\quad t\geq 0.$$

Like the full Euler equations, the BBM-BBM system possesses classical solitary
wave solutions. These are waves that travel with constant phase speed and without
change in their shape. Usually, both $\eta$ and $u$ profiles of solitary waves
have the shape of a $\sech^2$-function \cite{Chen2000,BC1998,DM2008}.
There are no known analytical formulas for classical solitary
waves of the BBM-BBM system (\ref{eq:bbmbbm}), although
there are analytical formulas describing other travelling wave solutions of
(\ref{eq:bbmbbm}); see \cite{Chen1998}.  The latter travelling waves are unphysical, as
the value of $\eta$ is negative in places so that the water surface extends below the
impenetrable bottom. Other travelling wave solutions with analytical formulas
assume that there is no water in the domain. One such formula, which is
useful mainly for testing numerical methods, can be written in the special
case $D=1$, $g=1$ in the form
\begin{align}\label{eq:exactsol}
\begin{aligned}
\eta_{\pm}(x,t) & = \frac{15}{4}\left(\cosh\left(3\sqrt{\frac{2}{5}}\left(x\mp\frac{5}{2}t\right) \right) -2\right)\sech^4\left(\frac{3}{\sqrt{10}}\left(x\mp\frac{5}{2}t\right) \right),\\
u_{\pm}(x,t) &= \frac{15}{2} \sech^2 \left(\frac{3}{\sqrt{10}}\left(x\mp\frac{5}{2}t\right) \right).
\end{aligned}
\end{align}
This travelling wave solution was derived in \cite{Chen1998, Chen2000} and is
apparently unstable under even tiny perturbations, which cause a point blow-up
phenomenon \cite{BC2016}. For generalisations, physical derivation, and a review
of the theory and numerical analysis, we refer to \cite{DM2008, KDFM2018}.

For practical situations and for numerical computations the Cauchy problem for the system
(\ref{eq:bbmbbm}) is usually not very useful, since numerical methods and practical
problems are set in bounded domains. The theory for several initial-boundary value
problems in a bounded interval $I=[a,b]$ has been established \cite{BC1998,ADM2009}.
More specifically, given appropriately smooth initial data $\eta(x,0)=\eta_0(x)$
and $u(x,0)=u_0(x)$, the existence and uniqueness of classical and weak solutions
were proved for the following problems:
\begin{itemize}
\item Wave-maker boundary conditions (Dirichlet--Dirichlet): $\eta(a,t)=h_1(t)$, $\eta(b,t)=h_2(t)$ and $u(a,t)=v_1(t)$, $u(b,t)=v_2(t)$,
\item Reflection boundary conditions (Neumann--Dirichlet): $\eta_x(a,t)=\eta_x(b,t)=0$ and $u(a,t)=u(b,t)=0$,
\item Periodic boundary conditions (Periodic--Periodic): $\partial_x^i\eta(a,t)=\partial_x^i\eta(b,t)$ and $\partial_x^iu(a,t)=\partial_x^iu(b,t)$, for all $i\geq 0$,  where $\partial_x^i$ denotes the $i$-th order partial derivative with respect to $x$.
\end{itemize}

The conservation of the quantities $\mathcal{M}$, $\mathcal{I}$, $\mathcal{H}$ and $\mathcal{E}$ on bounded intervals depends upon the choice of the boundary conditions. For example, mass (\ref{eq:mass1}) and energy (\ref{eq:energy}) are conserved
with periodic and reflecting boundary conditions, but not with
wave-maker conditions.
\begin{table}[ht!]
\caption{Conservative quantities for different initial-boundary value problems} \label{tab:conserv}
\begin{tabular}{ccccc}
\hline
& $\mathcal{M}(t;\eta)$ & $\mathcal{I}(t;u)$ & $\mathcal{H}(t;\eta,u)$ & $\mathcal{E}(t;\eta,u)$\\ \hline
Cauchy & $\checkmark$ & $\checkmark$ & $\checkmark$ & $\checkmark$\\
Dirichlet &  &  & $\checkmark$ &    \\
Reflective & $\checkmark$ & & & $\checkmark$ \\
Periodic & $\checkmark$ & $\checkmark$ & $\checkmark$ & $\checkmark$ \\
\hline
\end{tabular}
\end{table}
On the other hand, the momentum (\ref{eq:mass2})
is conserved by the solutions of the periodic boundary value problem only.
Finally, the conservation of the impulse functional (\ref{eq:h1}) is only valid in the
case of periodic and homogenous Dirichlet boundary value problems. The
conservation properties of the various initial-boundary value problems are
summarised in Table \ref{tab:conserv}.

The conservation of mass $\mathcal{M}$, especially, is one of the basic requirements for acceptable numerical simulations. Lack of its preservation is related to large wave damping. Conservation of the momentum $\mathcal{I}$ comes as a secondary requirement; note that momentum is conserved for the Cauchy problem, but not in the presence of a reflecting boundary.  Typically, conventional methods preserve $\mathcal{M}$ and $\mathcal{I}$ (when the exact solution preserves them) without sophisticated modifications. The conservation of the impulse $\mathcal{H}$ can enhance the numerical results, but (like the momentum) it is not conserved in the presence of a reflecting boundary.  The impulse is also conserved by conventional numerical methods such as the standard Galerkin method and by geometric integrators. The total energy $\mathcal{E}$ is not conserved by conventional methods, and yet its conservation is very important for both physical and geometrical reasons as it is the Hamiltonian of the particular system \cite{BCS2002}.  Note that in the exact solution the total energy {\em is} conserved even in the presence of reflecting boundaries (cf. Table \ref{tab:conserv}).

Several methods have been devised for the numerical solution of relevant
initial-boundary value problems for the system (\ref{eq:bbmbbm}). These include,
for example, spectral \cite{XRVA2018, PD2001}, finite volume/difference
\cite{DKM2011,BC1998}  and Galerkin finite element methods
\cite{ADM2010, DMS2010ii}, where they have been analysed, tested and studied in
depth.   Some of these numerical methods appear to have good conservation and accuracy
properties but none of them has been designed to conserve the energy functional
(\ref{eq:energy}). 
Conservative fully-discrete schemes based on the multi-symplectic structure of symmetric versions of Boussinesq systems were studied in \cite{DDM2019i,DDM2019ii} where the energy functionals contain only
quadratic nonlinearities. The energy functional of the BBM-BBM system though
contains cubic nonlinearities and thus multi-symplectic methods cannot be applied
in a straightforward manner.

The spatial semidiscretisation of the BBM-BBM system (\ref{eq:bbmbbm})
usually leads to a non-stiff system of ordinary differential equations.
In the case of finite element methods, non-stiffness was proven in
\cite{ADM2010} while for other spatial discretisations it was apparent after numerical experimentation.
For this reason, explicit Runge--Kutta methods with favourable stability properties
are efficient for its integration in time. The disadvantage of the explicit
Runge--Kutta methods tested in the previously mentioned works is that they do not
conserve the energy functional. Therefore, the fully-discrete systems proposed
so far for the numerical solution of the system (\ref{eq:bbmbbm}) are not
conservative with the exception of the very recent work  \cite{RMK2020}, where
appropriate collocation methods were considered for the spatial discretisation
combined with recently proposed relaxation Runge--Kutta methods for integration
in time \cite{KE2019, RSDPK2020, RLK2020}.

Because conservation of energy can be important, especially in some long-time
simulations, our focus is on energy preserving fully-discrete schemes.
In this paper we present a new Galerkin finite element method
with Lagrange elements that conserves the energy functional (\ref{eq:energy}).
Although it has been shown that the standard Galerkin finite element method based on
high-order splines has excellent conservation properties, that method is immensely hard to
use in two space dimensions and again it does not preserve the total energy $\mathcal{E}$. On the other hand, Lagrange elements can be extended and used
easily in two space dimensions, and for this reason it is important to acquire
the appropriate background of such conservative methods.

The Galerkin semidiscretisation of the new method is expressed and implemented in
the framework of mixed finite element methods \cite{BoBreFor13}. The resulting
system of ordinary differential equations, which is apparently non-stiff, is
integrated in time using explicit relaxation Runge--Kutta schemes
\cite{KE2019, RSDPK2020, RK2020} that are specially designed as interpolation
Runge--Kutta methods \cite{Z1986} to respect the conservation of the energy of
systems of ordinary differential equations. The resulting full discretisation is
referred to in what follows as the conservative Galerkin method.

Due to the non-existence of analytical formulas for classical solitary waves, we
test the method using classical solitary wave profiles generated numerically
with the Petviashvili method \cite{Petv1976}. The Petviashvili
method in the context of finite element methods for Boussinesq systems was first
presented in \cite{KMS2020}. In this paper we revisit the particular method in
one space dimension and for the analogous initial-boundary value problems.

The structure of this paper is the following: In Section \ref{sec:sfem} we present
the standard and the conservative Galerkin finite element methods, and review their
conservation properties.
Section \ref{sec:rrk} is dedicated to the temporal discretisation of these Galerkin
methods with explicit Runge--Kutta methods.
A detailed study of the convergence of the conservative Galerkin
method is presented in Section \ref{sec:ecr}. This study reveals an unexpected
convergence behaviour of mixed Galerkin methods for the specific problem, not
previously detected in standard Galerkin method.  Due to the significant
importance of solitary waves, Section \ref{sec:exper} presents numerical
experiments related to the propagation, reflection and interaction of solitary
waves, showing the behaviour of the conservative Galerkin method.
For the sake of completeness we present the Petviashvili method in the context
of Galerkin finite element methods in Appendix \ref{sec:appendix}. Conclusions
and perspectives are summarised in Section \ref{sec:conclu}.

\section{Galerkin finite element methods}\label{sec:sfem}

\subsection{Notation}

Let $a=x_0<x_1<\cdots < x_{N} = b$ be a uniform partition of the interval
$\bar{I}=[a,b]$ with mesh length $\Delta x=x_{i+1}-x_i$. Given an integer
$r\geq 1$, we shall consider the   finite element spaces of Lagrange elements
\begin{align}\label{eq:spaces}
\mathcal{P}^r:=\{\phi\in C(\bar{I}): ~ \phi|_{[x_j,x_{j+1}]}\in\mathbb{P}^{r}, ~ 0\leq j\leq N-1\},
\end{align}
where $C(I)$ is the usual space of continuous functions on $I$.
Moreover, $\mathbb{P}^r$ denotes the space of polynomials of degree at most $r$.
We also consider the Lagrange finite element spaces
$$\mathcal{P}^r_0:=\{\phi\in\mathcal{P}^r~|~ \phi(a)=\phi(b)=0\},$$
and the space of periodic Lagrange elements
$$\mathcal{P}^r_p:=\{\phi\in\mathcal{P}^r~|~ \phi(a)=\phi(b)\}.$$
In what follows we will denote by $H^s(I)$ the usual Sobolev spaces of
$s$-order weakly differentiable functions, and their norm by $\|\cdot\|_s$.
It is well-known that the Lagrange finite element spaces are all finite-dimensional
subspaces of the space $H^1(I)$.
The space of square-integrable functions will be denoted
by $L^2(I)=H^0(I)$ with the associated norm $\| \cdot \|$. We will also denote the spaces $\mathcal{P}^r_0$ and
$\mathcal{P}^r_p$ as $\mathcal{P}^r_b$ where $b$ can be either $0$ or $p$,
depending on the choice of the boundary conditions of the problem at hand.

The $L^2$-projection (or orthogonal projection) of a function $u\in L^2(I)$
onto the space $\mathcal{P}^r_b$  is defined as the operator
$P^r_b:L^2\to \mathcal{P}^r_b$ such that
\begin{align}
(P^r_b [u],\phi)=(u,\phi),\quad\mbox{ for all } \phi\in \mathcal{P}^r_b,
\end{align}
where
$$(u,v):=\int_a^b u\cdot v ~\dd x,$$
is the usual $L^2$-inner product.

\subsection{Standard Galerkin semidiscretisations}\label{sec:sgal}

We consider here the semidiscretisation of the initial-boundary value problem
written in nondimensional but unscaled variables:
\begin{align}\label{eq:ibvp1}
\begin{aligned}
\eta_t+[(1+\eta)u]_x-\frac{1}{6}\eta_{xxt}&=0,\\
u_t + \eta_x + uu_x-\frac{1}{6}u_{xxt}&=0,\\
\eta(x,0)=\eta_0(x),\quad u(x,0)&=u_0(x),\\
\end{aligned}
\end{align}
with either reflecting boundary conditions (Neumann--Dirichlet problem)
\begin{align}\label{eq:bc1}
\eta_x(a,t)=\eta_x(b,t)=0 \mbox{ and } u(a,t)=u(b,t)=0,
\end{align}
or periodic boundary conditions (Periodic--Periodic problem)
\begin{align}\label{eq:bc2}
\partial_x^i \eta(a,t)=\partial_x^i \eta_x(b,t)=0 \mbox{ and } \partial_x^i u(a,t)=\partial_x^i u(b,t),\quad \mbox{ for all } i=0,1,2,\ldots.
\end{align}

The standard Galerkin semidiscretisation of (\ref{eq:ibvp1}), (\ref{eq:bc1}) is
defined as follows: Let $[0,T]$ be a  nonempty closed subinterval of the maximal interval of existence and
uniqueness of the solution of (\ref{eq:ibvp1}) in $H^1(I)$. We seek
$\teta:[0, T]\to \mathcal{P}^r$ and $\tu:[0,T]\to \mathcal{P}^r_0$ such that
\begin{alignat}{2}\label{eq:sd1}
\begin{aligned}
(\teta_t,\chi)+\frac{1}{6}(\teta_{xt},\chi_x)&=\left( (1+\teta)\tu,\chi_x\right), &&\quad \mbox{ for all } \chi\in \mathcal{P}^r, \\
(\tu_t,\psi)+\frac{1}{6}(\tu_{xt},\psi_x) &= \frac{1}{2}(\tu^2,\psi_x)+(\teta,\psi_x), &&\quad \mbox{ for all } \psi\in \mathcal{P}^r_0,
\end{aligned}
\end{alignat}
with initial conditions
\begin{align}
\teta(x,0)=P^r[\eta_0(x)], \quad \tu(x, 0)=P_0^r [u_0(x)] .
\end{align}

The semidiscretisation of the initial-periodic boundary value problem is very
similar, the only difference being that the finite element spaces should be
replaced by the periodic polynomial space $\mathcal{P}^r_p$. In particular,
the semidiscrete solutions of the periodic problem are functions $\teta:[0, T]\to \mathcal{P}^r_p$ and
$\tu:[0,T]\to \mathcal{P}^r_p$ such that
\begin{alignat}{2}\label{eq:sd2}
\begin{aligned}
(\teta_t,\chi)+\frac{1}{6}(\teta_{xt},\chi_x)&=\left( (1+\teta)\tu,\chi_x\right), &&\quad \mbox{ for all } \chi\in \mathcal{P}^r_p, \\
(\tu_t,\psi)+\frac{1}{6}(\tu_{xt},\psi_x) &= \frac{1}{2}(\tu^2,\psi_x)+(\teta,\psi_x), &&\quad \mbox{ for all } \psi\in \mathcal{P}^r_p,
\end{aligned}
\end{alignat}
with initial conditions
\begin{align}
\teta(x,0)=P^r_p[\eta_0(x)], \quad \tu(x, 0)=P^r_p [u_0(x)] .
\end{align}

Both semidiscretisations can be found in \cite{ADM2009, DMS2010ii, ADM2010} where optimal
order convergence was proven. In particular, the numerical solution of these
semidiscretisations is known to converge  with order $r+1$ and the following 
error bound was shown to hold:
\begin{align}
\|\teta(t)-\eta(t)\|+\|\tu(t)-u(t)\|\leq C\Delta x^{r+1}, \qquad
\mbox{for all $t \in [0,T]$},
\end{align}
for some constant $C$ independent of $\Delta x$ \cite{DMS2010ii,ADM2010}.  These
semidiscretisations guarantee the conservation of the mass functional
$\mathcal{M}(t,\teta)$ but not the energy functional (\ref{eq:energy}).  In the
case of periodic boundary conditions, these semidiscretisations also conserve the
momentum $\mathcal{I}(t,\tu)$ and the impulse functional $\mathcal{H}(t;\teta,\tu)$.

\subsection{A conservative semidiscretisation for the reflective-boundary value problem}\label{sec:fem1}

Consider now the following modified Galerkin semidiscretisation for the
initial-boundary value problem with reflecting boundary conditions
(\ref{eq:ibvp1}), (\ref{eq:bc1}).  Let $\teta:[0, T]\to \mathcal{P}^r$ and
$\tu:[0,T]\to \mathcal{P}^r_0$ be the solutions of the system
\begin{alignat}{2}\label{eq:csd1}
\begin{aligned}
(\teta_t,\chi)+\frac{1}{6}({P^r_0[\teta_{x}]}_t,\chi_x)&=\left( P^r_0[(1+\teta)\tu],\chi_x\right), &&\quad \mbox{ for all $\chi\in \mathcal{P}^r$}, \\
(\tu_t,\psi)+\frac{1}{6}({P^r[\tu_{x}]}_t,\psi_x) &= \frac{1}{2}(P^r[\tu^2],\psi_x)+(P^r[\teta],\psi_x), &&\quad \mbox{ for all $ \psi\in \mathcal{P}^r_0$},
\end{aligned}
\end{alignat}
given the initial conditions
\begin{align}\label{eq:ic1}
\teta(x,0)=P^r[\eta_0(x)], \quad \tu(x, 0)=P_0^r [u_0(x)] .
\end{align}
This modified Galerkin semidiscretisation is derived from the standard Galerkin method
after replacing the dispersive and nonlinear terms with appropriate
$L^2$-projections in the associated finite element spaces. The
semidiscretisation (\ref{eq:csd1}) preserves both conserved quantities of the original problem
as they are noted in Table \ref{tab:conserv}. Indeed, the following proposition holds. 
\begin{proposition}\label{prop:reflcons}
The solution $(\teta,\tu)$ of (\ref{eq:csd1}), (\ref{eq:ic1}) satisfies the following conservation laws for all $t\in[0,T]$:
\begin{enumerate}
\item[(i)] Conservation of mass
$$\frac{\dd}{\dd t}\mathcal{M}(t;\teta)=0;$$
\item[(ii)] Conservation of energy
$$\frac{\dd}{\dd t}\mathcal{E}(t;\teta,\tu)=0.$$
\end{enumerate}
\end{proposition}
\begin{proof}
The conservation of mass (i) can be obtained by taking $\chi=1$ in (\ref{eq:csd1}).
To prove (ii), let
\begin{align}
& R:=-(1+\teta)\tu+\frac{1}{6}\teta_{xt},\\
& Q:=-\frac{1}{2}\tu^2-\teta+\frac{1}{6}\tu_{xt}.
\end{align}
Then we have
$$
\begin{aligned}
(\teta_t,Q)+(\tu_t,R) &= (\teta_t,P^r[Q])+(\tu_t, P^r_0[R]) \\
&= -(P^r_0[R],P^r[Q]_x)-(P^r[Q],P^r_0[R]_x) \\
&= 0.
\end{aligned}
$$
Furthermore,
$$
\begin{aligned}
(\teta_t,Q)+(\tu_t,R) &= (\teta_t,-\frac{1}{2}\tu^2-\teta+\frac{1}{6}\tu_{xt})+(\tu_t,-(1+\teta)\tu+\frac{1}{6}\teta_{xt})\\
&= \frac{1}{6}((\teta_t\tu_t)_x,1)-(\teta_t,\frac{1}{2}\tu^2+\teta)-(\tu_t,(1+\teta)\tu)\\
&= -\frac{1}{2}\left(\frac{\dd}{\dd t} \teta^2,1 \right)-\frac{1}{2}\left(\frac{\dd}{\dd t}\teta,\tu^2\right)-\frac{1}{2}\left(\frac{\dd}{\dd t}\tu^2, 1+\teta\right)\\
&=  -\frac{1}{2}\left(\frac{\dd}{\dd t}( \teta^2+(1+\teta)\tu^2),1\right)\\
&= -\frac{1}{2}\frac{\dd}{\dd t}\int_a^b \teta^2+(1+\teta)\tu^2\dd x .
\end{aligned}
$$
In the previous calculations we used the fact that $\tu\in \mathcal{P}^r_0$ and thus $\tu(a)=\tu(b)=0$, and the result follows.
\end{proof}

\begin{remark}
The projection in the first (mass) equation of (\ref{eq:csd1}) does not
need to be onto $\mathcal{P}_0^r$ but may instead be taken onto $\mathcal{P}^r$ without any
complications. The advantage of projection onto $\mathcal{P}^r_0$ is that
the Neumann boundary conditions of the free-surface elevation are guaranteed to be satisfied.
\end{remark}

The semidiscretisation (\ref{eq:csd1}) appears to be complicated, especially if
one considers its implementation. An alternative mixed finite element formulation
can resolve this difficulty by introducing new auxiliary variables as follows.
Let  $\teta:[0, T]\to \mathcal{P}^r$ and $\tu:[0,T]\to \mathcal{P}^r_0$ as before.
In addition we consider  $\tw, \tf:[0, T]\to \mathcal{P}^r_0$ and $\tv,\tg:[0,T]\to \mathcal{P}^r$ such that
\begin{alignat}{2}\label{eq:mixed1}
\begin{aligned}
(\teta_t,\chi)+\frac{1}{6}(\tw_t,\chi_x)&=(\tilde{f},\chi_x), &&\quad \mbox{ for all } \chi\in \mathcal{P}^r, \\
\frac{1}{6}({\teta_{x}}_t,\phi)- \frac{1}{6}(\tw_t,\phi)&=0, &&\quad \mbox{ for all } \phi\in\mathcal{P}^r_0 ,\\
(\tf,\zeta)&=((1+\teta)\tu,\zeta), &&\quad \mbox{ for all } \zeta\in\mathcal{P}^r_0 , \\
(\tu_t,\psi)+\frac{1}{6}(\tv_t,\psi_x) &= (\tg,\psi_x), &&\quad \mbox{ for all } \psi\in \mathcal{P}^r_0, \\
\frac{1}{6}({\tu_{x}}_t,\xi)-\frac{1}{6}(\tv_t,\xi)&=0, &&\quad \mbox{ for all } \xi\in\mathcal{P}^r , \\
(\tg,\theta)&=\left(\frac{1}{2}\tu^2+\teta,\theta\right), &&\quad \mbox{ for all } \theta\in\mathcal{P}^r ,
\end{aligned}
\end{alignat}
accompanied with the initial conditions
\begin{align}\label{eq:bcmixed1}
\teta(x,0)=P^r[\eta_0(x)], \quad \tw(x,0)=P^r_0[{\eta_0}_x(x)], \quad \tu(x, 0)=P^r_0 [u_0(x)], \quad \tv(x, 0)=P^r [{u_0}_x(x)] .
\end{align}
The formulations (\ref{eq:csd1}) and (\ref{eq:mixed1}) are equivalent, but
in the mixed formulation the $L^2$-projections are written
explicitly as unknowns. This is a characteristic that
makes mixed formulations very successful \cite{BoBreFor13}.
Note that the coefficients $1/6$ in the second and fifth equations have been
introduced so that the respective bilinear form resembles a saddle point problem, the solvability of which has been analysed in \cite{BoBreFor13}.

\subsection{A conservative semidiscretisation for the periodic boundary value problem}\label{sec:fem2}

Now we state a modified Galerkin
semidiscretisation for the initial-periodic boundary value problem
(\ref{eq:ibvp1})--(\ref{eq:bc2}).  Let $\teta:[0, T]\to \mathcal{P}^r_p$ and
$\tu:[0,T]\to \mathcal{P}^r_p$ be the solutions of the system
\begin{alignat}{2}\label{eq:csd2}
\begin{aligned}
(\teta_t,\chi)+\frac{1}{6}({P^r_p[\teta_{x}]}_t,\chi_x)&=( P^r_p[(1+\teta)\tu],\chi_x), &&\quad \mbox{ for all } \chi\in \mathcal{P}^r_p, \\
(\tu_t,\psi)+\frac{1}{6}({P^r_p[\tu_{x}]}_t,\psi_x) &= \frac{1}{2}(P^r_p[\tu^2],\psi_x)+(P^r_p[\teta],\psi_x), &&\quad \mbox{ for all } \psi\in \mathcal{P}^r_p,
\end{aligned}
\end{alignat}
given the initial conditions
\begin{align}\label{eq:ic2}
\teta(x,0)=P^r_p[\eta_0(x)], \quad \tu(x, 0)=P^r_p [u_0(x)] .
\end{align}
Similarly to the reflection case the following proposition holds in the case of periodic boundary conditions:
\begin{proposition}\label{prop:percons}
The solution $(\teta,\tu)$ of (\ref{eq:csd2}), (\ref{eq:ic2}) satisfies the following conservation laws for all $t\in[0,T]$:
\begin{enumerate}
\item[(i)] Conservation of mass 
$$\frac{\dd}{\dd t}\mathcal{M}(t;\teta)=0;$$
\item[(ii)] Conservation of momentum
$$\frac{\dd}{\dd t}\mathcal{I}(t;\tu)=0;$$
\item[(iii)] Conservation of energy
$$\frac{\dd}{\dd t}\mathcal{E}(t;\teta,\tu)=0.$$
\end{enumerate}
\end{proposition}
\begin{proof}
The proofs of (i) and (ii) are immediate by using $\chi=1$ and $\psi=1$ in
(\ref{eq:csd2}). The proof of (iii) follows exactly the same steps as in
Proposition \ref{prop:reflcons} and is omitted.
\end{proof}
\begin{remark}
The conservation of $\mathcal{H}(t;\teta,\tu)$ does not hold in the periodic case.
A numerical investigation of the accuracy of this conservation law will be one
of the subjects of the following sections.
\end{remark}

The mixed formulation with periodic conditions (\ref{eq:csd2}) is very similar
to that with reflective boundary conditions. Let  $\teta:[0, T]\to \mathcal{P}^r_p$ and
$\tu:[0,T]\to \mathcal{P}^r_p$ as before. In addition we consider
$\tw, \tf:[0, T]\to \mathcal{P}^r_p$ and $\tv,\tg:[0,T]\to \mathcal{P}^r_p$ such that
\begin{alignat}{2}\label{eq:mixed2}
\begin{aligned}
(\teta_t,\chi)+\frac{1}{6}(\tw_t,\chi_x)&=(\tilde{f},\chi_x), &&\quad \mbox{ for all } \chi\in \mathcal{P}^r_p, \\
\frac{1}{6}({\teta_{x}}_t,\phi)-\frac{1}{6}(\tw_t,\phi)&=0, &&\quad \mbox{ for all } \phi\in\mathcal{P}^r_p ,\\
(\tf,\zeta)&=((1+\teta)\tu,\zeta), &&\quad \mbox{ for all } \zeta\in\mathcal{P}^r_p , \\
(\tu_t,\psi)+\frac{1}{6}(\tv_t,\psi_x) &= (\tg,\psi_x), &&\quad \mbox{ for all } \psi\in \mathcal{P}^r_p, \\
\frac{1}{6}({\tu_{x}}_t,\xi)-\frac{1}{6}(\tv_t,\xi)&=0, &&\quad \mbox{ for all } \xi\in\mathcal{P}^r_p , \\
(\tg,\theta)&=\left(\frac{1}{2}\tu^2+\teta,\theta\right), &&\quad \mbox{ for all } \theta\in\mathcal{P}^r_p ,
\end{aligned}
\end{alignat}
accompanied with the initial conditions
\begin{align}\label{eq:bcmixed2}
\teta(x,0)=P^r_p[\eta_0(x)], \quad \tw(x,0)=P^r_p[{\eta_0}_x(x)], \quad \tu(x, 0)=P^r_p [u_0(x)], \quad \tv(x, 0)=P^r_p [{u_0}_x(x)] .
\end{align}

\section{Temporal discretisation}\label{sec:rrk}

Systems (\ref{eq:mixed1}) and (\ref{eq:mixed2}) are autonomous systems of
ordinary differential equations which, as we shall see later, are non-stiff. We
proceed with their temporal discretisation, which is chosen so that the energy
functional $\mathcal{E}$ is conserved in both cases. The time-integration methods
that we will use are known as relaxation Runge--Kutta methods (RRK)
and were introduced recently in \cite{KE2019, RSDPK2020} and studied in detail
in the case of Hamiltonian systems in \cite{RMK2020}. We proceed by providing a brief description of relaxation Runge--Kutta methods methods.

A system of ordinary differential equations, such as (\ref{eq:mixed1}) or (\ref{eq:mixed2}), can be written
\begin{align}
\begin{aligned}
\frac{\dd}{\dd t} \by(t) &= \bF(t,\by(t)), \quad t\in(0,T],\\
\by(0)&=\by_0 ,
\end{aligned}
\end{align}
where $\by(t)=[\by_i(t)]$ denotes the unknown vector function.
We consider a timestep $0<\Delta t<1$ and a uniform grid $0=t_0<t_1<\cdots <t_K=T$ with $t_{i+1}=t_i+(i+1)\Delta t$ for $i=0,1,\ldots,K-1$.  A general explicit Runge--Kutta method with $s$ stages can be fully described by its Butcher tableau
\renewcommand\arraystretch{1.2}
\begin{align}\label{eq:butcher}
\begin{array}{c|c}
c & A\\
\hline
& b^\mathrm{T}
\end{array},
\end{align}
where $A=[a_{ij}]_{i,j=1}^s$ is an $s\times s$ lower-triangular matrix with zeros
in the principal diagonal, and $b=[b_j]_{j=1}^s$ and $c=[c_j]_{j=1}^s$ are
$s$-dimensional vectors. If $\by^n$ stands for an approximation of $\by(t^n)$,
then an explicit Runge--Kutta method given by a Butcher tableau (\ref{eq:butcher})
can be expressed as
\begin{align}
\tilde{\by}^i&=\by^n+\Delta t\sum_{j=1}^{i-1} a_{ij}\,\bF(t_n+c_j\Delta t, \tilde{\by}^j), \quad i=1,2,\ldots,s, \label{eq:stages}\\
\by(t_n+\Delta t)\approx \by^{n+1}&=\by^n+\Delta t\sum_{i=1}^sb_i \,\bF(t_n+c_i\Delta t,\tilde{\by}^i). \label{eq:updates}
\end{align}
If
$$\bd^n=\sum_{i=1}^s b_i \,\bF_i,$$
with $\bF_i=\bF(t_n+c_i\Delta t,\tilde{\by}^i)$, then the respective
relaxation Runge--Kutta method is formulated by the replacement of the update
formula (\ref{eq:updates}) with an update in the same direction as the previous
formula but of a different length:
\begin{align}\label{eq:rrkupdate}
\by(t_n+\gamma^n\Delta t)\approx \by_\gamma^{n+1}=\by^n+\gamma^n\Delta t\ \bd^n,
\end{align}
where $\by^n=\by^{n}_\gamma$, for $n=1,2,\ldots$.
The parameter $\gamma^n$ is called the relaxation parameter and is defined
appropriately for all $n$ in order for the method to satisfy certain properties.
For example, we can choose $\gamma^n$ so that
$$\mathcal{E}(t^{n}+\gamma^n\Delta t; \by^{n+1}_\gamma)=\mathcal{E}(t^n;\by^n),$$
where $\mathcal{E}$ is the energy functional (\ref{eq:energy}).
Note that the prescribed time increment $\Delta t$ is replaced by $\gamma^n \Delta t$ in
the solution update formula \eqref{eq:updates} to obtain \eqref{eq:rrkupdate}, but the original $\Delta t$
is still used in the computation of the stages \eqref{eq:stages}.  As long as the prescribed
time step is admissible from the point of view of stability, the new step size $\Delta t^n_\gamma$ will be very close to the prescribed value $\Delta t$.

In our case we have $\by^n=(H^n, W^n, U^n, V^n)^\mathrm{T}$, where
$H^n$, $W^n$, $U^n$ and $V^n$ denote the fully discrete approximations in the appropriate
finite element spaces of $\teta(\cdot,t^n)$, $\tw(\cdot,t^n)$, $\tu(\cdot,t^n)$
and $\tv(\cdot,t^n)$, respectively.
Denote also the fully discrete approximations
of (\ref{eq:rrkupdate}) $\teta(\cdot,t^n+\gamma^n\Delta t)$,
$\tw(\cdot,t^n+\gamma^n\Delta t)$, $\tu(\cdot,t^n+\gamma^n\Delta t)$ and
$\tv(\cdot,t^n+\gamma^n\Delta t)$ in the appropriate finite element spaces by
$H^{n+1}_\gamma$, $W^{n+1}_\gamma$, $U^{n+1}_\gamma$ and $V^{n+1}_\gamma$,
respectively. Then, we can express the updates in $H^n$ and $U^n$ as
\begin{align}\label{eq:nupdt}
\begin{aligned}
H^{n+1}_\gamma &= H^n+\gamma^n\Delta t\ \bd^n_\eta,\\
U^{n+1}_\gamma &= U^n+\gamma^n\Delta t\ \bd^n_u,
\end{aligned}\quad \mbox{ for } n=0,1,2,\ldots,
\end{align}
where $\bd^n_\eta$ is the part of $\bd^n$ corresponding to the respective
solution $H^n$ and  $\bd^n_u$ is the part of $\bd^n$ corresponding to the
respective solution $U^n$. To devise a conservative time integration scheme,
we choose $\gamma^n$ to be such that
\begin{align}\label{eq:disenergy}
\mathcal{E}(t^{n}+\gamma^n\Delta t;H^{n+1}_\gamma,U^{n+1}_\gamma)=\mathcal{E}(t^n;H^n,U^n).
\end{align}
The parameter $\gamma^n$ can be determined by solving the algebraic equation (\ref{eq:disenergy}) for each $n$. Substitution of (\ref{eq:nupdt}) into (\ref{eq:disenergy}) leads to the following roots: $\gamma_n^0=0$ or
\begin{align}\label{eq:gamman1}
\gamma_n^\pm \cdot \Delta t =\frac{-\Beta\pm\sqrt{\Beta^2-4\Alpha\Gamma}}{2\Alpha},
\end{align}
where
\begin{align}\label{eq:gamman2}
\begin{aligned}
\Alpha&:= \int_a^b\bd^n_\eta(\bd^n_u)^2\dd x,\\
\Beta&:= \int_a^b (\bd^n_\eta)^2+(1+H^n)(\bd^n_u)^2+2U^n\bd^n_\eta \bd^n_u \dd x, \\
\Gamma&:= \int_a^b[2H^n+(U^n)^2]\bd^n_\eta+2 U^n(1+H^n)\bd^n_u\dd x.
\end{aligned}
\end{align}
Assume that our explicit Runge--Kutta method has global error of order $\Delta t^p$.  According to \cite{RSDPK2020,RLK2020} it is expected that there is a root $\gamma^n$ such that
\begin{equation}\label{eq:bounds1}
|\gamma^n-1|=O(\Delta t^{p-1})\ .
\end{equation}
This means that by choosing appropriate $\Delta t\ll 1$ we expect $\gamma^n$ to be close to $1$. Unfortunately, the floating-point evaluation of analytical expressions for the roots of the cubic polynomial
equation (\ref{eq:disenergy}) using (\ref{eq:gamman1}) can be affected by
catastrophic cancellation, and thus the numerical
approximation of these roots is suggested. In our case we use the the secant method
for their approximation (with convergence tolerance $10^{-10}$ for the absolute error). Alternative root finding procedures for relaxation methods are
discussed in \cite{RDP2020,RMK2020}.
The temporal discretisation of system (\ref{eq:mixed1}) is presented in
Algorithm \ref{alg:timem}, while the time-marching scheme for the periodic
system (\ref{eq:mixed2}) is very similar; for this reason, we omit its detailed
description.

\begin{algorithm}
\begin{algorithmic}
\State Given the timestep $\Delta t$ and interval of integration $[0,T]$
\State $t^n=0$
\State $H^0=P^r[\eta_0(x)]$ and $U^0=P^r_0[u_0(x)]$
\While{$t^n\leq T$}
\For{$i =1,2,\ldots, s$}
\State $\tilde{H}^i=H^n+\Delta t\sum_{j=1}^{i-1}a_{ij}H^{n,j}$
\State $\tilde{U}^i=U^n+\Delta t\sum_{j=1}^{i-1}a_{ij}U^{n,j}$
\State $\tilde{f}=P^r_0 [(1+\tilde{H}^i)\tilde{U}^i]$
\State $\tilde{g}=P^r \left[\frac{1}{2}(\tilde{U}^i)^2 +\tilde{H}^i\right]$
\State $\left\{ \begin{array}{ll} (H^{n,i},\chi)+\frac{1}{6}(W^{n,i},\chi_x)=(\tilde{f},\chi_x), & \forall \chi \in\mathcal{P}^r, \\
\frac{1}{6}(H^{n,i}_x,\phi)-\frac{1}{6}(W^{n,i},\phi)=0, & \forall \phi \in\mathcal{P}^r_0,\end{array}\right.$ evaluated at $t^{n,i}=t^{n}+c_i\Delta t$
\State $\left\{ \begin{array}{ll} (U^{n,i},\psi)+\frac{1}{6}(V^{n,i},\psi_x)=(\tilde{g},\psi_x), &  \forall \psi \in\mathcal{P}^r_0, \\
\frac{1}{6}(U^{n,i}_x,\xi)-\frac{1}{6}(V^{n,i},\xi)=0, &  \forall \xi \in\mathcal{P}^r,  \end{array}\right.$ evaluated at $t^{n,i}=t^{n}+c_i\Delta t$
\EndFor
\State $\bd^n_\eta=\sum_{i=1}^s{b_i}H^{n,i} $
\State $\bd^n_u=\sum_{i=1}^s{b_i}U^{n,i} $
\State Compute $\gamma^n$ by solving numerically (\ref{eq:disenergy})
\State $H^{n+1}=H^n+\gamma^n\Delta t \bd^n_\eta$
\State $U^{n+1}=U^n+\gamma^n\Delta t \bd^n_u$
\State $t^{n+1}=t^n+\gamma^n\Delta t$
\State $n\gets n+1$
\EndWhile
\end{algorithmic}
\caption{Conservative fully-discrete scheme for the regularised shallow water equations}\label{alg:timem}
\end{algorithm}

\begin{remark} Taking $\chi=1$ in Algorithm \ref{alg:timem} we observe that
$$\int_a^b H^{n,i} \dd x=0, \quad \mbox{for all } i=1,2,\ldots,s \quad \mbox{ and }\quad  n=0,1,2,\ldots,$$ and thus
$$
\int_a^b\bd^n_\eta \dd x=0, \quad \mbox{for all}\quad n=0,1,2,\ldots .
$$
Therefore, the fully-discrete scheme in addition to the energy $\mathcal{E}(t;H^n,U^n)$ preserves also the mass $\mathcal{M}(t;H^{n})$, as expected.
\end{remark}

Apparently, the resulting systems of ordinary differential equations at
(\ref{eq:sd1}) and (\ref{eq:sd2}) are not stiff and thus explicit Runge--Kutta
methods can be efficient in long-time numerical simulations. The new
mixed finite element semidiscretisations result also in non-stiff systems of
ordinary differential equations and thus in this paper we use the
conservative relaxation methods of \cite{KE2019,RSDPK2020,RK2020} for the
integration in time of (\ref{eq:mixed1}) and (\ref{eq:mixed2}).  Note
that in our case the semidiscrete systems are autonomous, meaning that
$\bF(t,\by)=\bF(\by)$, but we consider here the more general case to include
the possibility of inhomogeneous equations and the consideration of other
source terms.

\section{Experimental convergence rates and accuracy estimates}\label{sec:ecr}

In this section we study experimentally the convergence and accuracy of the new
numerical scheme. First, we study the spatial errors and convergence rates.
In order to complete the experimental study of the error analysis we perform a
detailed study using certain error indicators pertinent to the propagation of
travelling wave solutions. For the temporal discretisation of both
initial-boundary value problems we use a relaxation Runge--Kutta method
related to the classical four-stage, fourth-order Runge--Kutta method given by
the Butcher tableau:
\begin{align}\label{eq:RK4}
\begin{array}{c|c}
c & A\\
\hline
& b^\mathrm{T}
\end{array} ~=~ \begin{array}{c|cccc}
0 &  0 & 0 & 0  & 0  \\
1/2 &  1/2 & 0  & 0 & 0 \\
1/2 & 0 & 1/2 & 0 & 0 \\
1 & 0 & 0 & 1 & 0 \\
\hline
& 1/6 & 1/3 & 1/3 & 1/6
\end{array}.
\end{align}
For the numerical implementation of the fully-discrete schemes we employed the library Fenics \cite{Fenics}.

\subsection{Spatial convergence rates for the Neumann--Dirichlet problem}

We consider the initial-boundary value problem  (\ref{eq:ibvp1}), (\ref{eq:bc1}) in the interval $[a,b]=[0,1]$.
Due to the lack of exact analytical solutions for the specific problem, we consider the artificial solution of the form
\begin{align}
\begin{aligned}
\eta(x,t)=\mathrm{e}^{2t}\cos(\pi x),\\
u(x,t)=\mathrm{e}^tx\sin(\pi x),
\end{aligned}
\end{align}
for $x\in[0,1]$ and $t\in[0,1]$. This particular solution satisfies a
non-homogenous initial-boundary value problem of the form
(\ref{eq:ibvp1})--(\ref{eq:bc2}) with appropriate right-hand sides in both the
mass and momentum conservation equations. Since the energy is not conserved under these
non-homogenous equations, we use the classical
Runge--Kutta method of order four (without relaxation) with small timestep $\Delta t= \Delta x/10$ to
ensure that the numerical errors introduced by the time integration are
negligible compared to the errors due to the spatial discretisation.

We study the spatial convergence of the fully-discrete scheme using the errors defined as
\begin{align}\label{eq:errorf}
E_s[F]:=\|F(\cdot,T;\Delta x)-F_{\text{exact}}(\cdot,T)\|_s,
\end{align}
where $F=F(\cdot;\Delta x)$ is the computed solution in $\mathcal{P}^r$, i.e. either $H\approx \eta(\cdot,T)$ or $U\approx u(\cdot,T)$, $F_{\text{exact}}$ is the corresponding exact solution and $\|\cdot\|_s$ with $s=0,1$ corresponds to the $L^2$ and $H^1$ norms, respectively. In addition to these standard errors, we also compute the errors of the quantities $W\approx w(\cdot,T)=\eta_x(\cdot,T)$ and $V\approx v(\cdot,T)=u_x(\cdot,T)$. These errors are practically $H^1$-norm based errors since the new numerical method computes approximations of the first derivatives $\eta_x$ and $u_x$ at every time step. In particular we compute the error
\begin{align}\label{eq:modnorm}
\begin{aligned}
\tilde{E}_1[H]&:=\sqrt{\|H(\cdot,T;\Delta x)-\eta(\cdot,T)\|^2+\|W(\cdot,T;\Delta x)-\eta_x(\cdot,T)\|^2}, \\
\tilde{E}_1[U]&:=\sqrt{\|U(\cdot,T;\Delta x)-u(\cdot,T)\|^2+\|V(\cdot,T;\Delta x)-u_x(\cdot,T)\|^2}.
\end{aligned}
\end{align}
Throughout this work we use quadrature rules that are exact for the polynomial
integrands, with the choice of quadrature depending on the finite element
spaces used for the spatial discretisation.

We compute the fully-discrete numerical solutions $H$ and $U$ until $T=1$, and the errors (\ref{eq:errorf}) when $\mathcal{P}^r$ is with $r=1,2,3$ and $4$, and for
\begin{align}\label{eq:deltax}
\Delta x\in\{\Delta x_0,\Delta x_1,\ldots,\Delta x_5\} = \{0.1,0.05,0.02,0.01,0.005\}.
\end{align}
All these errors are presented in Tables \ref{tab:T1}--\ref{tab:T4}. We also list the experimental convergence rate:
\begin{align}\label{eq:ecr}
R_s[F]:=\frac{\ln(E_s[F(\cdot;\Delta x_{k-1})]/E_s[F(\cdot;\Delta x_k)])}{\ln(\Delta x_{k-1}/\Delta x_k)},
\end{align}
where $\Delta x_k$ stands for the grid size listed in row $k$ in each of the Tables \ref{tab:T1}--\ref{tab:T4}. The convergence rates based on the error $\tilde{E}_1$ are denoted by $\tilde{R}_1$ and are presented in Table \ref{tab:onenorm}.

\begin{table}[htbp]
  \centering
  \caption{Spatial errors and rates of convergence based on $E_0$ and $E_1$ errors for $r=1$ (Neumann--Dirichlet problem, $\Delta t= \Delta x/10$)}
    \begin{tabular}{c|cccc|cccc}\hline
   $\Delta x$ & $E_0[H]$ &  $R_0[H] $ &  $E_0[U]$ &  $R_0[U] $ & $E_1[H]$ &  $R_1[H] $ &  $E_1[U]$ &  $R_1[U] $ \\ \hline
  $0.100$  &  $5.595\times 10^{-2} $ &     --        & $ 1.527\times 10^{-2} $ &     --         &  $1.226\times 10^{0} $ &     --          & $ 1.717\times 10^{-1} $ &     --       \\
  $0.050$  &  $1.378\times 10^{-2} $ &  $2.021$ & $ 3.832\times 10^{-3} $ &  $1.994$ &  $5.966\times 10^{-1} $ &  $1.039$ & $ 7.343\times 10^{-2} $ &  $1.226$\\
  $0.020$  &  $2.197\times 10^{-3} $ &  $2.004$ & $ 6.143\times 10^{-4} $ &  $1.998$ &  $2.369\times 10^{-1} $ &  $1.008$ & $ 2.696\times 10^{-2} $ &  $1.093$\\
  $0.010$  &  $5.490\times 10^{-4} $ &  $2.001$ & $ 1.537\times 10^{-4} $ &  $1.999$ &  $1.183\times 10^{-1} $ &  $1.001$ & $ 1.314\times 10^{-2} $ &  $1.037$\\
  $0.005$  &  $1.372\times 10^{-4} $ &  $2.000$ & $ 3.842\times 10^{-5} $ &  $2.000$ &  $5.915\times 10^{-2} $ &  $1.000$ & $ 6.489\times 10^{-3} $ &  $1.018$\\
 \hline
\end{tabular}
  \label{tab:T1}%
\end{table}%

\begin{table}[htbp]
  \centering
  \caption{Spatial errors and rates of convergence based on $E_0$ and $E_1$ errors for $r=2$ (Neumann--Dirichlet problem, $\Delta t= \Delta x/10$)}
    \begin{tabular}{c|cccc|cccc}\hline
   $\Delta x$ & $E_0[H]$ &  $R_0[H] $ &  $E_0[U]$ &  $R_0[U] $ & $E_1[H]$ &  $R_1[H] $ &  $E_1[U]$ &  $R_1[U] $ \\ \hline
  $0.100$  &  $8.117\times 10^{-3} $ &     --         & $9.375\times 10^{-4} $ &     --         &  $6.287\times 10^{-1} $ &     --         &$ 7.105\times 10^{-2} $ &     --       \\
  $0.050$  &  $2.022\times 10^{-3} $ &  $2.005$ & $2.266\times 10^{-4} $ &  $2.049$ &  $3.133\times 10^{-1} $ &  $1.005$ & $ 3.491\times 10^{-2} $ &  $1.025$\\
  $0.020$  &  $3.233\times 10^{-4} $ &  $2.001$ & $3.590\times 10^{-5} $ &  $2.011$ &  $1.252\times 10^{-1} $ &  $1.001$ & $ 1.389\times 10^{-2} $ &  $1.006$\\
  $0.010$  &  $8.081\times 10^{-5} $ &  $2.000$ & $8.962\times 10^{-6} $ &  $2.002$ &  $6.259\times 10^{-2} $ &  $1.000$ & $ 6.940\times 10^{-3} $ &  $1.001$\\
  $0.005$  &  $2.020\times 10^{-5} $ &  $2.000$ & $2.240\times 10^{-6} $ &  $2.001$ &  $3.130\times 10^{-2} $ &  $1.000$ & $ 3.469\times 10^{-3} $ &  $1.000$\\
 \hline
\end{tabular}
  \label{tab:T2}%
\end{table}%

\begin{table}[htbp]
  \centering
  \caption{Spatial errors and rates of convergence based on $E_0$ and $E_1$ errors for $r=3$ (Neumann--Dirichlet problem, $\Delta t= \Delta x/10$)}
    \begin{tabular}{c|cccc|cccc}\hline
   $\Delta x$ & $E_0[H]$ &  $R_0[H] $ &  $E_0[U]$ &  $R_0[U] $ & $E_1[H]$ &  $R_1[H] $ &  $E_1[U]$ &  $R_1[U] $ \\ \hline
  $0.100$  &  $7.335\times 10^{-5} $ &     --         & $7.812\times 10^{-6} $  &     --         & $9.383\times 10^{-3} $ &     --         & $ 6.437\times 10^{-4} $ &     --       \\
  $0.050$  &  $4.508\times 10^{-6} $ &  $4.024$ & $4.865\times 10^{-7} $  &  $4.005$ &  $1.154\times 10^{-3} $ &  $3.023$ & $ 7.699\times 10^{-5} $ &  $3.064$\\
  $0.020$  &  $1.149\times 10^{-7} $ &  $4.005$ & $1.246\times 10^{-8} $  &  $3.999$ &  $7.353\times 10^{-5} $ &  $3.005$ & $ 4.835\times 10^{-6} $ &  $3.021$\\
  $0.010$  &  $7.176\times 10^{-9} $ &  $4.001$ & $7.794\times 10^{-10}$ &  $3.999$ &  $9.186\times 10^{-6} $ &  $3.001$ & $ 6.012\times 10^{-7} $ &  $3.008$\\
  $0.005$  &  $4.482\times 10^{-10} $& $4.001$ & $4.867\times 10^{-11}$ &  $4.001$ &  $1.148\times 10^{-6} $ &  $3.001$ & $ 7.496\times 10^{-8} $ &  $3.004$\\
 \hline
\end{tabular}
  \label{tab:T3}%
\end{table}%

\begin{table}[htbp]
  \centering
  \caption{Spatial errors and rates of convergence based on $E_0$ and $E_1$ errors for $r=4$ (Neumann--Dirichlet problem, $\Delta t= \Delta x/10$)}
    \begin{tabular}{c|cccc|cccc}\hline
   $\Delta x$ & $E_0[H]$ &  $R_0[H] $ &  $E_0[U]$ &  $R_0[U] $ & $E_1[H]$ &  $R_1[H] $ &  $E_1[U]$ &  $R_1[U] $ \\ \hline
  $0.100$  &  $8.590\times 10^{-6} $   &     --         & $2.146\times 10^{-7} $  &     --         & $1.629\times 10^{-3} $ &     --         & $ 3.698\times 10^{-5} $ &     --       \\
  $0.050$  &  $5.347\times 10^{-7} $  &  $4.006$ & $1.215\times 10^{-8}$  &   $4.143$ &  $2.029\times 10^{-4} $ &  $3.005$ & $ 4.487\times 10^{-6} $ &  $3.043$\\
  $0.020$  &  $1.367\times 10^{-8} $   &  $4.001$ & $3.014\times 10^{-10}$  & $4.034$ &  $1.297\times 10^{-5} $ &  $3.001$ & $ 2.847\times 10^{-7} $ &  $3.010$\\
  $0.010$  &  $8.531\times 10^{-10} $ &  $4.003$ & $1.875\times 10^{-11}$ &  $4.007$ &  $1.619\times 10^{-6} $ &  $3.003$ & $ 3.553\times 10^{-8} $ &  $3.002$\\
 \hline
\end{tabular}
  \label{tab:T4}%
\end{table}%

\begin{table}[htbp]
\caption{Spatial errors and rates of convergence based on $\tilde{E}_1$ error (Neumann--Dirichlet problem, $\Delta t= \Delta x/10$)}
\centering
\begin{tabular}{c|cccc|cccc}\hline
   $\Delta x$ & $\tilde{E}_1[H]$ &  $\tilde{R}_1[H] $ &  $\tilde{E}_1[U]$ &  $\tilde{R}_1[U] $ & $\tilde{E}_1[H]$ &  $\tilde{R}_1[H] $ &  $\tilde{E}_1[U]$ &  $\tilde{R}_1[U] $ \\ \hline
 & \multicolumn{4}{c|}{$r=1$} & \multicolumn{4}{c}{$r=2$}   \\ \hline
  $0.100$  &  $ 2.012\times 10^{-1} $   &     --       & $6.414\times 10^{-2} $ &     --        & $4.421\times 10^{-2} $ &     --         & $ 1.064\times 10^{-2} $ &     --       \\
  $0.050$  &  $ 4.945\times 10^{-2} $ &  $2.024$ & $1.592\times 10^{-2}$  & $2.010$ &  $1.111\times 10^{-2} $ &  $1.992$ & $  2.677\times 10^{-3} $ &  $1.990$\\
  $0.020$  &  $ 7.891\times 10^{-3} $ &  $2.003$ & $2.546\times 10^{-3}$  & $2.001$ &  $1.782\times 10^{-3} $ &  $1.998$ & $ 4.300\times 10^{-4} $ &  $1.996$\\
  $0.010$  &  $ 1.973\times 10^{-3} $ &  $2.000$ & $6.365\times 10^{-4}$ &  $2.000$ &  $4.457\times 10^{-4} $ &  $1.999$ & $ 1.076\times 10^{-4} $ &  $1.999$\\
  $0.005$  &  $ 4.934\times 10^{-4} $&   $2.000$ & $1.591\times 10^{-4}$ &  $2.000$ &  $1.114\times 10^{-4} $ &  $2.000$ & $ 2.690\times 10^{-5} $ &  $2.000$\\
 \hline
 & \multicolumn{4}{c|}{$r=3$} & \multicolumn{4}{c}{$r=4$}   \\ \hline
  $0.100$  &  $ 5.358\times 10^{-4} $   &     --       & $8.736\times 10^{-5} $ &     --        &   $6.485\times 10^{-5} $ &     --         &   $ 1.020\times 10^{-5} $ &     --       \\
  $0.050$  &  $ 3.209\times 10^{-5} $ &  $4.062$ & $5.046\times 10^{-6}$  &   $4.114$ &  $4.103\times 10^{-6} $ &    $3.982$ & $ 6.605\times 10^{-7} $ &   $3.949$\\
  $0.020$  &  $ 8.159\times 10^{-7} $ &  $4.007$ & $1.268\times 10^{-7}$  &   $4.021$ &  $1.055\times 10^{-7} $ &    $3.996$ & $ 1.713\times 10^{-8} $ &   $3.986$\\
  $0.010$  &  $ 5.102\times 10^{-8} $ &  $3.999$ & $7.903\times 10^{-9}$  &   $4.004$ &  $6.593\times 10^{-9} $ &    $4.000$ & $ 1.073\times 10^{-9} $ &   $3.997$\\
  $0.005$  &  $ 3.191\times 10^{-9} $&   $3.999$ & $4.935\times 10^{-10}$ &  $4.001$ &  $4.139\times 10^{-10} $ &  $3.994$ & $ 6.750\times 10^{-11} $ &  $3.990$\\
 \hline
\end{tabular}
\label{tab:onenorm}
\end{table}%

It is easily observed that the convergence rates for odd values of $r$ are
optimal, as expected. On the other hand, the convergence rates for even $r$ are
suboptimal. More specifically, we conjecture the following rule for the
convergence rates:
$$
\begin{aligned}
& R^{2k-1}_0[H]=R^{2k-1}_0[U]=\tilde{R}^{2k-1}_1[H]=\tilde{R}^{2k-1}_1[U]=2k,\quad R^{2k-1}_1[H]=R^{2k-1}_1[U]=2k-1 ,\\
& R^{2k}_0[H]=R^{2k}_0[U]=\tilde{R}^{2k}_1[H]=\tilde{R}^{2k}_1[U]=2k,\quad R^{2k}_1[H]=R^{2k}_1[U]=2k-1,
\end{aligned}  \quad \text{ for } k=1,2,\ldots.
$$
The respective errors are smaller for larger values of $r$ and for the same grids,
while the convergence rates are the same. It is also noted that when $r=4$ the
numerical method with $\Delta t=0.01$ reaches its limits with errors of the order
of $10^{-12}$ and thus no further convergence can be observed in double-precision floating point arithmetic.

A similar even/odd dichotomy is known to occur for discontinuous Galerkin approximations to linear second-order elliptic partial differential equations (cf. \cite{MR2511735}) and for first-order linear hyperbolic equations (cf. \cite{liu2020sub}) on uniform meshes. In the context of dispersion analysis on uniform computational meshes, an even/odd dichotomy is reported in \cite{A2014} for both continuous and discontinuous Galerkin approximations of a first-order linear hyperbolic equation in one space dimension, and in \cite{MR2576525} for a linear second-order wave equation in one space dimension. For a continuous Galerkin approximation of the regularised shallow water wave system considered here, the observed even/odd dichotomy is perhaps new and unusual.

On the other hand, the convergence rates using the modified $H^1$-error norms
(\ref{eq:modnorm}) appear to be super-optimal for odd and optimal for even
values of $r$. In particular, it can be observed that
$$
\tilde{R}^{2k-1}_1[H]=\tilde{R}^{2k-1}_1[U]=\tilde{R}^{2k}_1[H]=\tilde{R}^{2k}_1[U]=2k , \text{ for }\quad k=1,2,\ldots.
$$
The observed super-approximation property is promising from the point of view of further
developments of the specific conservative method.

\subsection{Spatial convergence rates for the periodic problem}

Similarly to the Neumann--Dirichlet boundary value problem we consider a
non-homogenous initial-periodic boundary value problem of the form (\ref{eq:ibvp1}),
(\ref{eq:bc2}) admitting the exact solution
\begin{align}
\begin{aligned}
\eta(x,t)&=\mathrm{e}^{t}\sin(2\pi (x-2t)),\\
u(x,t)&=\mathrm{e}^{t/2}\sin(2\pi (x-t/2),
\end{aligned}
\end{align}
in $[0,1]$ with the appropriate right-hand sides. We integrate the
initial-periodic boundary value problem with the semidiscretisation
(\ref{eq:mixed2}), (\ref{eq:bcmixed2}) using  Algorithm \ref{alg:timem} adapted to
periodic boundary conditions up to $T=1$, and we record the errors $E_s[H]$,
$E_s[U]$, for $s=0,1$, and $\tilde{E}_1[H]$ and $\tilde{E}_1[U]$. The corresponding
experimental rates of convergence computed using (\ref{eq:ecr}), (\ref{eq:deltax})
are very similar to those obtained in the Neumann--Dirichlet case but with some
exceptions. In the case of piecewise linear elements in $\mathcal{P}^1$ the
rates in the $H^1$-norm are improved, presumably because of the replacement of the Dirichlet boundary conditions with periodic boundary conditions. The experimental convergence rates are presented in Tables \ref{tab:T6}--\ref{tab:T9}.

\begin{table}[htbp]
  \centering
  \caption{Spatial errors and rates of convergence  based on $E_0$ and $E_1$ errors  for $r=1$ (Periodic problem, $\Delta t= \Delta x/10$)}
    \begin{tabular}{c|cccc|cccc}\hline
   $\Delta x$ & $E_0[H]$ &  $R_0[H] $ &  $E_0[U]$ &  $R_0[U] $ & $E_1[H]$ &  $R_1[H] $ &  $E_1[U]$ &  $R_1[U] $ \\ \hline
  $0.100$  &  $6.310\times 10^{-2} $ &     --        &  $ 7.875\times 10^{-2} $ &     --         & $2.298\times 10^{-0} $ &     --        & $ 2.793\times 10^{-0} $ &     --       \\
  $0.050$  &  $1.579\times 10^{-2} $ &  $1.998$ & $ 1.930\times 10^{-2} $ &  $2.028$ &  $1.109\times 10^{-0} $ &  $1.051$ & $ 1.345\times 10^{-0} $ &  $1.054$\\
  $0.020$  &  $2.529\times 10^{-3} $ &  $1.999$ & $ 3.072\times 10^{-3} $ &  $2.006$ &  $4.390\times 10^{-1} $ &  $1.011$ & $ 5.325\times 10^{-1} $ &  $1.011$\\
  $0.010$  &  $6.323\times 10^{-4} $ &  $2.000$ & $ 7.673\times 10^{-4} $ &  $2.001$ &  $2.192\times 10^{-1} $ &  $1.002$ & $ 2.659\times 10^{-1} $ &  $1.002$\\
  $0.005$  &  $1.581\times 10^{-4} $ &  $2.000$ & $ 1.918\times 10^{-4} $ &  $2.000$ &  $1.095\times 10^{-1} $ &  $1.001$ & $ 1.329\times 10^{-1} $ &  $1.001$\\
 \hline
\end{tabular}
  \label{tab:T6}%
\end{table}%

\begin{table}[htbp]
  \centering
  \caption{Spatial errors and rates of convergence  based on $E_0$ and $E_1$ errors for $r=2$ (Periodic problem, $\Delta t= \Delta x/10$)}
    \begin{tabular}{c|cccc|cccc}\hline
   $\Delta x$ & $E_0[H]$ &  $R_0[H] $ &  $E_0[U]$ &  $R_0[U] $ & $E_1[H]$ &  $R_1[H] $ &  $E_1[U]$ &  $R_1[U] $ \\ \hline
  $0.100$  &  $3.957\times 10^{-3} $ &     --        &  $ 8.803\times 10^{-3} $ &     --         & $3.221\times 10^{-1} $ &     --        & $  6.550\times 10^{-1} $ &     --       \\
  $0.050$  &  $9.034\times 10^{-4} $ &  $2.131$ & $ 2.229\times 10^{-3} $ &  $1.982$ &  $1.425\times 10^{-1} $ &  $1.177$ & $ 3.413\times 10^{-1} $ &  $0.941$\\
  $0.020$  &  $1.405\times 10^{-4} $ &  $2.031$ & $ 3.581\times 10^{-4} $ &  $1.995$ &  $5.459\times 10^{-2} $ &  $1.047$ & $ 1.384\times 10^{-1} $ &  $0.985$\\
  $0.010$  &  $3.498\times 10^{-5} $ &  $2.006$ & $ 8.958\times 10^{-5} $ &  $1.999$ &  $2.712\times 10^{-2} $ &  $1.009$ & $ 6.935\times 10^{-2} $ &  $0.997$\\
  $0.005$  &  $8.736\times 10^{-6} $ &  $2.002$ & $ 2.240\times 10^{-5} $ &  $2.000$ &  $1.354\times 10^{-2} $ &  $1.002$ & $ 3.470\times 10^{-2} $ &  $0.999$\\
 \hline
\end{tabular}
  \label{tab:T7}%
\end{table}%

\begin{table}[htbp]
  \centering
  \caption{Spatial errors and rates of convergence  based on $E_0$ and $E_1$ errors for $r=3$ (Periodic problem, $\Delta t= \Delta x/10$)}
    \begin{tabular}{c|cccc|cccc}\hline
   $\Delta x$ & $E_0[H]$ &  $R_0[H] $ &  $E_0[U]$ &  $R_0[U] $ & $E_1[H]$ &  $R_1[H] $ &  $E_1[U]$ &  $R_1[U] $ \\ \hline
  $0.100$  &  $7.548\times 10^{-5} $ &     --          &  $ 9.877\times 10^{-5} $ &     --          & $1.156\times 10^{-2} $ &     --         & $1.450\times 10^{-2} $ &     --       \\
  $0.050$  &  $4.689\times 10^{-6} $   &  $4.009$ & $ 5.806\times 10^{-6} $  &  $4.088$ &  $1.418\times 10^{-3} $ &  $3.027$ & $1.735\times 10^{-3} $ &  $3.063$\\
  $0.020$  &  $1.199\times 10^{-7} $   &  $4.001$ & $ 1.461\times 10^{-7} $  &  $4.019$ &  $9.031\times 10^{-5} $ &  $3.005$ & $1.097\times 10^{-4} $ &  $3.013$\\
  $0.010$  &  $7.494\times 10^{-9} $   &  $4.000$ & $ 9.109\times 10^{-9} $  &  $4.004$ &  $1.128\times 10^{-5} $ &  $3.001$ & $1.369\times 10^{-5} $ &  $3.002$\\
  $0.005$  &  $4.686\times 10^{-10} $ &  $3.999$ & $ 5.713\times 10^{-10} $ & $3.995$ &  $1.410\times 10^{-6} $ &  $3.000$ & $1.710\times 10^{-6} $ &  $3.001$\\
 \hline
\end{tabular}
  \label{tab:T8}%
\end{table}%

\begin{table}[htbp]
  \centering
  \caption{Spatial errors and rates of convergence  based on $E_0$ and $E_1$ errors for $r=4$ (Periodic problem, $\Delta t= \Delta x/10$)}
    \begin{tabular}{c|cccc|cccc}\hline
   $\Delta x$ & $E_0[H]$ &  $R_0[H] $ &  $E_0[U]$ &  $R_0[U] $ & $E_1[H]$ &  $R_1[H] $ &  $E_1[U]$ &  $R_1[U] $ \\ \hline
  $0.100$  &  $4.639\times 10^{-6} $ &     --          &  $ 1.149\times 10^{-5} $ &     --          & $9.062\times 10^{-4} $ &     --         & $2.121\times 10^{-3} $ &     --       \\
  $0.050$  &  $2.704\times 10^{-7} $   &  $4.101$ & $ 7.208\times 10^{-7} $  &  $3.994$ &  $1.035\times 10^{-4} $ &  $3.130$ & $2.715\times 10^{-4} $ &  $2.966$\\
  $0.020$  &  $6.776\times 10^{-9} $   &  $4.023$ & $ 1.848\times 10^{-8} $  &  $3.999$ &  $6.435\times 10^{-6} $ &  $3.032$ & $1.751\times 10^{-5} $ &  $2.992$\\
  $0.010$  &  $4.220\times 10^{-10} $ &  $4.005$ & $ 1.155\times 10^{-9} $  &  $4.000$ &  $8.005\times 10^{-7} $ &  $3.007$ & $2.191\times 10^{-6} $ &  $2.998$\\
  $0.005$  &  $2.872\times 10^{-11} $ &  $3.877$ & $ 6.951\times 10^{-11} $ &  $4.055$ &  $1.079\times 10^{-7} $ &  $2.891$ & $2.607\times 10^{-7} $ &  $3.071$\\
 \hline
\end{tabular}
  \label{tab:T9}%
\end{table}%

The convergence rates seem to follow the same pattern as in the case of
reflective boundary conditions, and thus for odd values of $r$ we observe the
expected optimal behaviour in the $L^2$-norm. In addition to these errors,
we also computed the errors and convergence rates for the homogenous periodic
boundary value problem (\ref{eq:ibvp1}), (\ref{eq:bc2}) using the known, exact
formula of the  travelling wave solution (\ref{eq:exactsol}) in the interval
$[a,b]=[-20,20]$. We integrate the specific initial-periodic boundary value problem
until $T=10$ using the relaxation Runge--Kutta method based on the explicit
four-stage, fourth-order Runge--Kutta method with tableau (\ref{eq:RK4}).
The time interval of the integration is adequate to allow the travelling wave to
complete one period, and thus to test the boundary conditions as well.
Our numerical simulation was stable for the whole time of integration even though
the specific travelling wave is known to be unstable \cite{BC2016}. The results
are almost identical to the non-homogenous case and the convergence rates
(not shown here) agree with the convergence rates we presented in Tables \ref{tab:T6}--\ref{tab:T9}.

We also computed the modified $H^1$-errors (\ref{eq:modnorm}). The experimental
convergence rates are presented in Table \ref{tab:T10} where we observe the same
pattern in the rates as with the previous boundary-value problem.

\begin{table}[htbp]
\caption{Spatial errors and rates of convergence based on $\tilde{E}_1$ error (Periodic problem, $\Delta t= \Delta x/10$)}
\centering
\begin{tabular}{c|cccc|cccc}\hline
   $\Delta x$ & $\tilde{E}_1[H]$ &  $\tilde{R}_1[H] $ &  $\tilde{E}_1[U]$ &  $\tilde{R}_1[U] $ & $\tilde{E}_1[H]$ &  $\tilde{R}_1[H] $ &  $\tilde{E}_1[U]$ &  $\tilde{R}_1[U] $ \\ \hline
 & \multicolumn{4}{c|}{$r=1$} & \multicolumn{4}{c}{$r=2$}   \\ \hline
  $0.100$  &  $ 3.954\times 10^{-1} $   &     --       & $4.800\times 10^{-1} $ &     --        & $1.394\times 10^{-2} $ &     --         & $ 6.711\times 10^{-2} $ &     --       \\
  $0.050$  &  $ 1.001\times 10^{-1} $ &  $1.982$ & $1.216\times 10^{-1}$  & $1.982$ &  $3.630\times 10^{-3} $ &  $1.941$ & $ 1.660\times 10^{-2} $ &  $2.015$\\
  $0.020$  &  $ 1.608\times 10^{-2} $ &  $1.995$ & $1.951\times 10^{-2}$  & $1.997$ &  $6.116\times 10^{-4} $ &  $1.944$ & $ 2.647\times 10^{-3} $ &  $2.004$\\
  $0.010$  &  $ 4.022\times 10^{-3} $ &  $1.999$ & $4.880\times 10^{-3}$ &  $1.999$ &  $1.543\times 10^{-4} $ &  $1.987$ & $ 6.615\times 10^{-4} $ &  $2.001$\\
  $0.005$  &  $ 1.006\times 10^{-3} $&   $2.000$ & $1.220\times 10^{-3}$ &  $2.000$ &  $3.865\times 10^{-5} $ &  $1.997$ & $ 1.653\times 10^{-4} $ &  $2.000$\\
 \hline
 & \multicolumn{4}{c|}{$r=3$} & \multicolumn{4}{c}{$r=4$}   \\ \hline
  $0.100$  &  $ 5.860\times 10^{-4} $   &     --       & $6.244\times 10^{-4} $ &     --        &   $5.967\times 10^{-5} $ &     --         &   $7.859\times 10^{-5} $ &     --       \\
  $0.050$  &  $ 3.071\times 10^{-5} $ &  $4.254$ & $3.663\times 10^{-5}$  &   $4.091$ &  $4.523\times 10^{-6} $ &    $3.722$ & $5.083\times 10^{-6} $ &   $3.951$\\
  $0.020$  &  $ 7.658\times 10^{-7} $ &  $4.029$ & $9.280\times 10^{-7}$  &   $4.011$ &  $1.227\times 10^{-7} $ &    $3.936$ & $1.319\times 10^{-7} $ &   $3.985$\\
  $0.010$  &  $ 4.772\times 10^{-8} $ &  $4.004$ & $5.793\times 10^{-8}$  &   $4.002$ &  $7.735\times 10^{-9} $ &    $3.988$ & $8.212\times 10^{-9} $ &   $4.005$\\
  $0.005$  &  $ 2.973\times 10^{-9} $ &  $4.005$ & $3.607\times 10^{-9}$ &    $4.005$ &  $4.939\times 10^{-10} $ &  $3.969$ & $6.131\times 10^{-10} $ & $3.744$\\
 \hline
\end{tabular}
\label{tab:T10}
\end{table}%

In addition to the experimental convergence rates and errors we recorded the energy
$\mathcal{E}(t^n;H^n,U^n)$ and the mass $\mathcal{M}(t^n; H^n)$. Both quantities
are conserved to within 15 decimal digits. Indicatively, we present the graphs of the
errors in mass, energy and the relaxation parameter in Figure \ref{fig:errors1} in
the cases of $r=1$ and $r=3$ with $\Delta x=0.25$. We observe
that the errors are always less than $10^{-13}$ in the conserved quantities, independent of
the choice of the degree of the piecewise polynomial in the finite element. The parameter $\gamma^n$ remains always
close to 1, indicating in general the good conservation properties of the specific
choice of Runge--Kutta method for the integration in time of such a problem.

\begin{figure}[h!]
  \centering
  \includegraphics[width=\columnwidth]{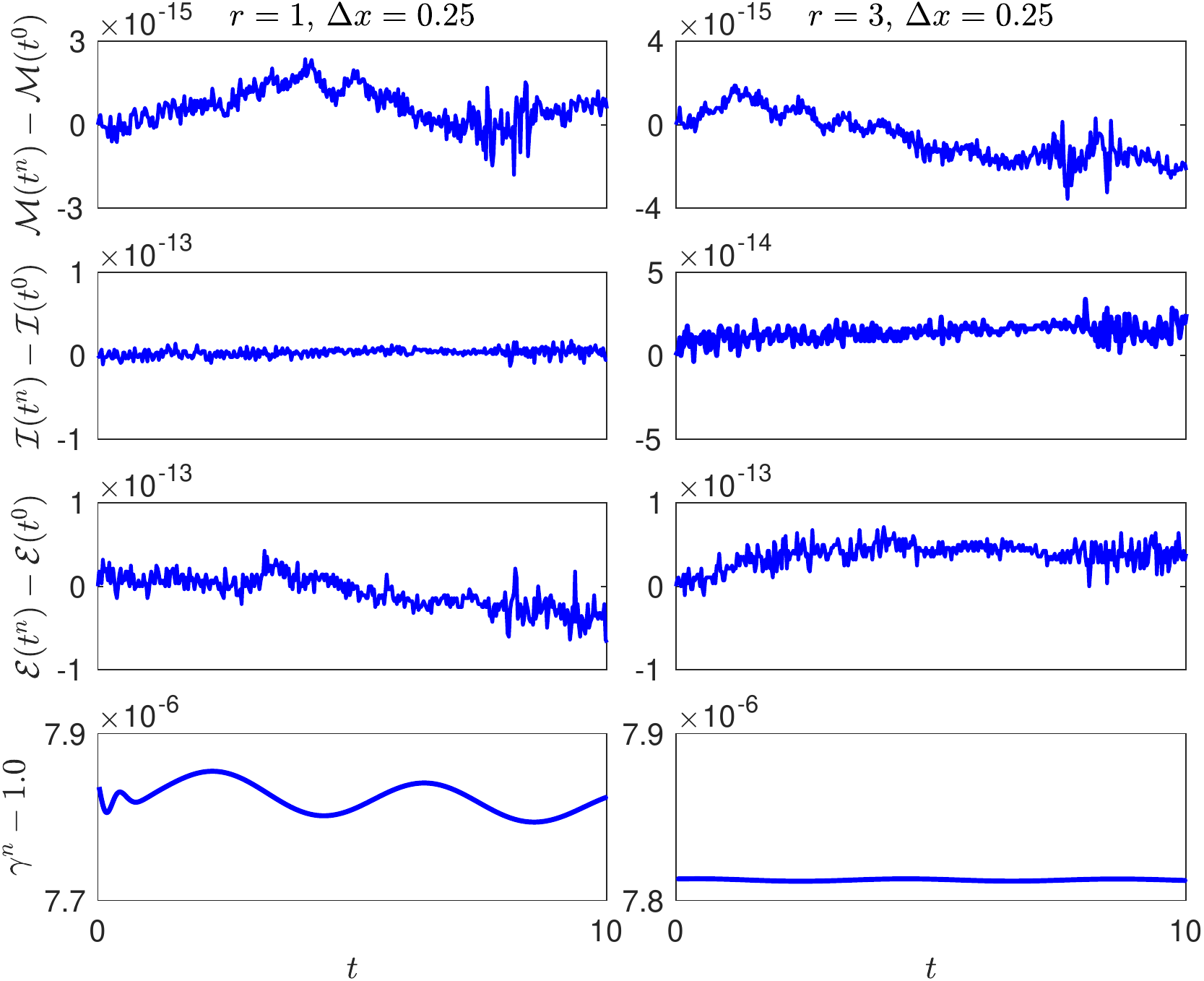}
  \caption{Errors in the mass, energy and relaxation parameters for the conservative Galerkin method with relaxation Runge--Kutta method of order 4, $\Delta t=\Delta x/10$}
  \label{fig:errors1}
\end{figure}
\begin{figure}[h!]
  \centering
  \includegraphics[width=\columnwidth]{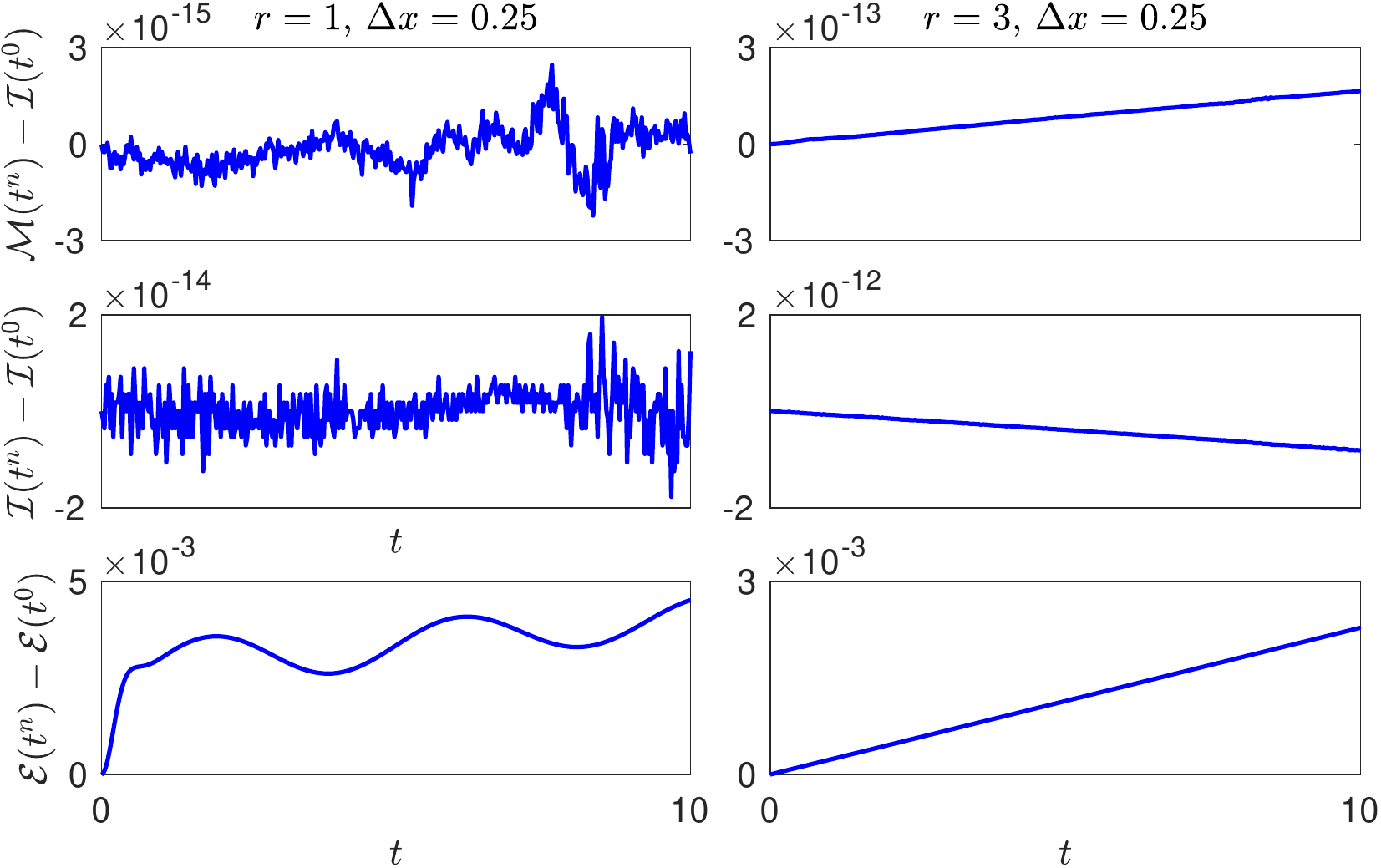}
  \caption{Errors in the mass and energy, standard Galerkin method with classical Runge--Kutta method of order 4, $\Delta t=\Delta x/10$}
  \label{fig:errors2}
\end{figure}

Figure \ref{fig:errors2} presents the error in mass and energy
obtained with the standard Galerkin method (\ref{eq:sd2}) and the classical
Runge--Kutta scheme (without relaxation). We observe that using the same
discretisation parameters as with the conservative method, the mass is conserved
but the energy is increasing with time. In both cases  the error in the
energy conservation is quite large with accuracy to within 3 digits.  In all cases where high-order elements were used, other
floating-point errors arise, resulting in relatively large errors in the conserved quantities
even though the errors in the solutions appear to be small.  Note that the
fully-discrete scheme based on the standard Galerkin method does not benefit from the
use of the relaxation Runge--Kutta method since the spatial discretisation does not
preserve the energy. Moreover, using the standard Galerkin method with cubic splines
for the spatial discretisation, the conservation properties are much better than
those reported here for both the conservative and non-conservative methods,
\cite{ADM2010}. In two-dimensional problems where Lagrange elements are practically
the only choice, the new method can be of significant importance.

\subsection{Errors in the propagation of travelling waves}

To further assess the efficiency of the new numerical method we study some error
indicators relevant to the propagation of travelling waves, namely, the amplitude,
shape and phase errors. For any timestep $t^n$ we define the (normalised)
\emph{amplitude error} as
\begin{align}
E_{\text{amp}}(t^n):=\frac{|H(x^\ast(t^n),t^n)-H_0|}{|H_0|},
\end{align}
where $H(x,t^n)=H^n$ is the fully discrete approximation of the free-surface
elevation $\eta$ at $t^n$, $H_0$ is the initial peak amplitude and $x^\ast(t^n)$
is the point where $H(\cdot,t^n)$ achieves its maximum. In order to find the value
$x^\ast(t^n)$ we solve the equation $\frac{\dd}{\dd x} H(x,t^n)=0$ using the bisection
method. (In the case of piecewise linear functions $x^\ast(t^n)$ is always a node
of the grid and so we do not use the bisection method). The bisection method is
initiated with an interval of length $10\cdot\Delta x$.

Having the value $x^\ast(t^n)$ computed, we define the \emph{phase error} that
measures the error between the numerical approximation and the theoretical value
of the phase speed $c_s$ of a solitary wave as
\begin{align}
E_{\text{phase}}:=|x^\ast(t^n)-c_s t^n|.
\end{align}
The (normalised) \emph{shape error} measures the difference in shape between the
numerical solution and the exact travelling wave translated appropriately for
best fit. The shape error with respect to the $L^2$ norm is defined as
\begin{align}
E_{\text{shape}}(t^n):=\min_s~\zeta(s),\qquad \zeta(s):=\frac{\|H(\cdot,t^n)-\eta(\cdot,s)\|}{\|\eta(\cdot,0)\|},
\end{align}
where the minimum of $\zeta(s)$ is found again using the bisection method for
solving the equation $\frac{\mathrm{d}}{\mathrm{d}s}\zeta^2 (s)=0$ in the vicinity of $t^n$.

Note that for the bisection method we use a convergence tolerance of $10^{-10}$
for the absolute error. For the computation of the $L^2$ norm we use the
Gauss--Legendre numerical quadrature with three nodes in each subinterval.

We study the propagation of a classical solitary wave with moderate speed
$c_s=\sqrt{1.6}$ in the interval $[-20,20]$ and with periodic boundary conditions
up to $T=1000$. We generated the specific solitary wave using the Petviashvili
method with cubic Lagrange elements $\mathcal{P}^3$ (see Appendix \ref{sec:appendix}). In the case of linear or quadratic Lagrange elements, we used the $L^2$-projection of the solitary wave and its first derivative onto the appropriate finite element space. The reason for using the $\mathcal{P}^3$ numerical solution for the travelling wave solution is because the accuracy of the derivative of the initial conditions can affect the accuracy of the conservative method. It also serves better for comparison purposes. The amplitude of the generated solitary wave is $A\approx 0.5919$.

We recorded the previously mentioned error indicators for piecewise linear,
quadratic and cubic Lagrange elements ($\mathcal{P}^r$, $r=1,2,3$). For these
experiments we consider a uniform spatial grid with $\Delta x=0.1$, while for the
integration in time we considered $\Delta t=\Delta x=0.1$. The
propagation of the specific solitary wave is stable even for larger values of
$\Delta t$. For example, the propagation of the solitary wave with $\mathcal{P}^1$
elements turns out to be unstable for $\Delta t>10\Delta x$ while it remains stable for $\Delta t\leq 10\Delta x$. This observation
indicates that the fully-discrete scheme remains in general stable without imposing restrictive stability conditions even for the
conservative Galerkin method.

\begin{figure}[h!]
  \centering
  \includegraphics[width=\columnwidth]{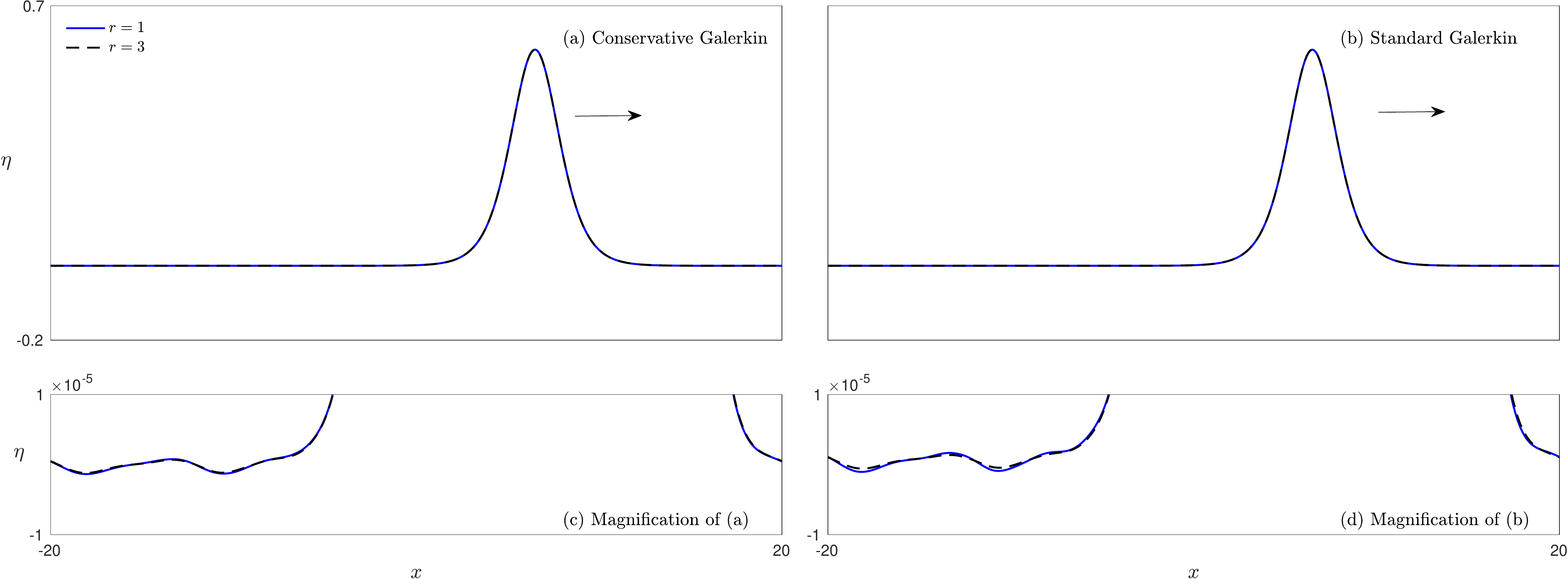}
  \caption{The propagation of a solitary wave with $c_s=\sqrt{1.6}$ and magnification of trailing tails. Periodic boundary conditions, $T=100$, ($\Delta x=\Delta t=0.1$).  Note the different $y$-axis scales in the lower figures.}
  \label{fig:solwave0}
\end{figure}
\begin{figure}[ht!]
  \centering
  \includegraphics[width=\columnwidth]{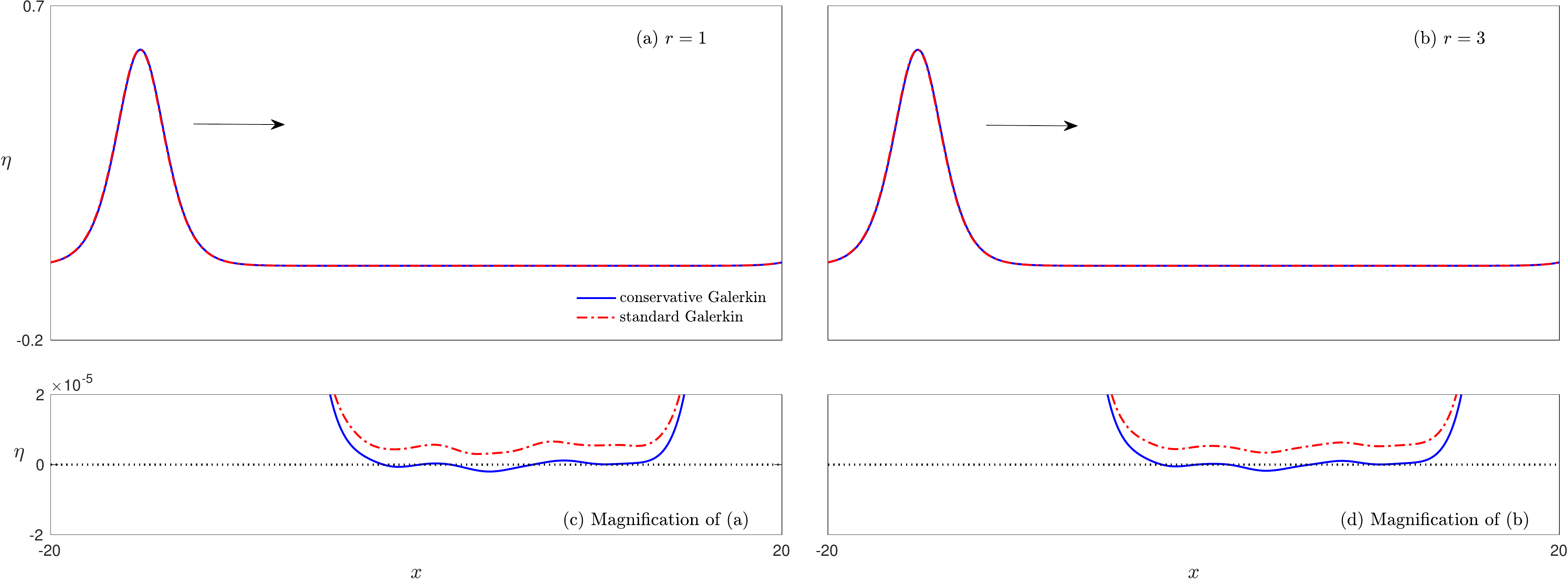}
  \caption{The propagation of a solitary wave with $c_s=\sqrt{1.6}$ and magnification of trailing tails. Periodic boundary conditions, $T=1000$, ($\Delta x=\Delta t=0.1$).}
  \label{fig:sw22}
\end{figure}

\begin{table}[ht!]
\caption{Mean value of $E_{amp}$, $E_{phase}$, $E_{shape}$ for $t\in[80,100]$ in the propagation of a classical solitary wave with speed $c_s=\sqrt{1.6}$, ($\Delta x=\Delta t$, $T=100$); (a) conservative Galerkin, (b) standard Galerkin }
\label{tab:errorsind}
\begin{tabular}{ l || l l l  | l l l  }
 \multicolumn{7}{l}{(a) Conservative Galerkin} \\
 \hline
 \multicolumn{1}{c}{} &  \multicolumn{3}{l}{$r=1$, piecewise linear polynomials} &  \multicolumn{3}{c}{$r=2$, piecewise quadratic polynomials} \\
 \hline
$\Delta x$ & $E_{amp}$ & $E_{phase}$ & $E_{shape}$ & $E_{amp}$ & $E_{phase}$ & $E_{shape}$  \\
 \hline
$0.1$ & $2.7351\times 10^{-4}$ & $2.4913\times 10^{-2}$ &  $1.8112\times 10^{-4}$ & $3.8778\times 10^{-5}$ & $1.5362\times 10^{-4}$ & $3.4944\times 10^{-5}$ \\
$0.05$ & $6.6823\times 10^{-5}$ &  $1.2492\times 10^{-2}$ & $4.5165\times 10^{-5}$ & $9.5895\times 10^{-6}$ & $1.8518\times 10^{-5}$ & $8.6722\times 10^{-6}$\\
$0.025$ & $1.6610\times 10^{-5}$ &  $6.2509\times 10^{-3}$ & $1.1296\times 10^{-5}$ & $2.4218\times 10^{-6}$ & $3.9029\times 10^{-6}$ & $2.2284\times 10^{-6}$ \\
 \hline
\end{tabular}
\begin{tabular}{ l || l l l}
 \multicolumn{1}{c}{} &  \multicolumn{3}{l}{$r=3$, piecewise cubic polynomials} \\
 \hline
$\Delta x$ & $E_{amp}$ & $E_{phase}$ & $E_{shape}$ \\
 \hline
$0.1$ & $8.1121\times 10^{-6}$ & $1.4479\times 10^{-4}$ & $9.0861\times 10^{-6}$  \\
$0.05$ & $4.7310\times 10^{-7}$ & $7.1440\times 10^{-6}$  & $6.2001\times 10^{-7}$ \\
$0.025$ & $2.3526\times 10^{-8}$ & $3.6515\times 10^{-7}$ & $3.4356\times 10^{-7}$ \\
 \hline
\end{tabular}

\begin{tabular}{ l || l l l  | l l l  }
\multicolumn{7}{l}{(b) Standard Galerkin} \\
\hline
 \multicolumn{1}{c}{} &  \multicolumn{3}{l}{$r=1$, piecewise linear polynomials} &  \multicolumn{3}{c}{$r=2$, piecewise quadratic polynomials} \\
 \hline
$\Delta x$ & $E_{amp}$ & $E_{phase}$ & $E_{shape}$ & $E_{amp}$ & $E_{phase}$ & $E_{shape}$ \\
 \hline
$0.1$ & $8.8025\times 10^{-4}$  & $2.5331\times 10^{-2}$ & $9.1723\times 10^{-4}$ & $2.2471\times 10^{-5}$ & $2.7796\times 10^{-4}$ & $1.8906\times 10^{-5}$ \\
$0.05$ & $6.7293\times 10^{-5}$ & $4.5788\times 10^{-2}$ & $5.5322\times 10^{-5}$ & $8.3509\times 10^{-7}$ & $9.7556\times 10^{-6}$ & $1.0418\times 10^{-6}$\\
$0.025$ & $1.6622\times 10^{-5}$ & $3.1229\times 10^{-2}$ & $1.3839\times 10^{-5}$ & $3.3815\times 10^{-8}$ & $5.2705\times 10^{-7}$ & $3.9631\times 10^{-7}$ \\
 \hline
\end{tabular}

\begin{tabular}{ l || l l l}
 \multicolumn{1}{c}{} &  \multicolumn{3}{l}{$r=3$, piecewise cubic polynomials} \\
 \hline
$\Delta x$ & $E_{amp}$ & $E_{phase}$ & $E_{shape}$ \\
 \hline
$0.1$ & $2.2249\times 10^{-5}$ & $1.8299\times 10^{-5}$ & $2.7805\times 10^{-4}$ \\
$0.05$ & $8.2113\times 10^{-7}$ & $9.7529\times 10^{-6}$ & $8.5842\times 10^{-7}$ \\
$0.025$ & $3.2946\times 10^{-8}$ & $4.2946\times 10^{-7}$ & $3.8875\times 10^{-7}$ \\
 \hline
\end{tabular}
\end{table}

As the solitary wave travels to the right, several oscillations are shed to the
left of the wave due to errors of the initial approximation and also due to error
introduced by the fully-discrete scheme. By the time $T=100$ the solitary wave has
crossed the boundaries three times. The solitary wave profile is shown in
Figures \ref{fig:solwave0}(a) and (b) for $T=100$. In these figures it is hard to observe differences between
the solutions. We are only able to observe differences and the generation of the various trailing tails in a magnification of the solitary wave at around the rest position of the free surface. These oscillations become smaller by
increasing the degree $r$ of the elements or by decreasing the mesh length
$\Delta x$. For the specific values of $\Delta x$ and $\Delta t$,
the magnitude of the tails is less than $10^{-5}$, which is much less than the
expected error of order $\Delta x^r$.

\begin{table}[ht!]
\caption{Conservation of invariants $\mathcal{M}$, $\mathcal{I}$, $\mathcal{H}$, $\mathcal{E}$,  in the propagation of a classical solitary wave with speed $c_s=\sqrt{1.6}$, ($\Delta x=\Delta t$, $T=100$); (a) conservative Galerkin, (b) standard Galerkin }
\label{tab:errorsinvars}
\begin{tabular}{l || l l l l}
\multicolumn{5}{l}{(a) Conservative Galerkin} \\
\hline
\multicolumn{1}{c}{} &  \multicolumn{4}{l}{$r=1$, piecewise linear polynomials} \\
 \hline
$\Delta x$ & $E_{\mathcal{M}}$ & $E_{\mathcal{I}}$ & $E_{\mathcal{H}}$ & $E_{\mathcal{E}}$ \\
 \hline
$0.1$ & $6.2172\times 10^{-15}$ & $4.2188\times 10^{-15}$ & $1.2656\times 10^{-7}$ & $1.8874\times 10^{-15}$ \\
$0.05$ & $1.4211\times 10^{-14}$ &  $9.3259\times 10^{-15}$ & $8.1135\times 10^{-9}$ & $3.3307\times 10^{-15}$ \\
$0.025$ & $1.5987\times 10^{-14}$ &  $8.2157\times 10^{-15}$ & $5.1039\times 10^{-10}$ & $4.7740\times 10^{-15}$ \\
 \hline
\end{tabular}
\begin{tabular}{l || l l l l}
\multicolumn{1}{c}{} &  \multicolumn{4}{l}{$r=2$, piecewise quadratic polynomials} \\
\hline
$\Delta x$ & $E_{\mathcal{M}}$ & $E_{\mathcal{I}}$ & $E_{\mathcal{H}}$ & $E_{\mathcal{E}}$ \\
\hline
$0.1$ &  $7.5495\times 10^{-15}$ & $4.2188\times 10^{-15}$ & $2.2095\times 10^{-6}$ & $1.8874\times 10^{-15}$ \\
$0.05$ &    $3.5083\times 10^{-14}$ & $2.2427\times 10^{-14}$ & $5.5608\times 10^{-7}$ & $3.2196\times 10^{-15}$\\
$0.025$ &  $5.9064\times 10^{-14}$ & $5.3069\times 10^{-14}$ & $1.3909\times 10^{-7}$ & $4.9960\times 10^{-15}$ \\
 \hline
\end{tabular}
\begin{tabular}{l || l l l l}
 \multicolumn{1}{c}{} &  \multicolumn{4}{l}{$r=3$, piecewise cubic polynomials} \\
 \hline
$\Delta x$ & $E_{\mathcal{M}}$ & $E_{\mathcal{I}}$ & $E_{\mathcal{H}}$ & $E_{\mathcal{E}}$ \\
 \hline
$0.1$ & $3.1974\times 10^{-14}$ & $2.9088\times 10^{-14}$ & $2.6408\times 10^{-11}$  & $2.8866\times 10^{-15}$  \\
$0.05$ & $5.6399\times 10^{-14}$ & $5.4179\times 10^{-14}$  & $1.1535\times 10^{-13}$ & $3.7748\times 10^{-15}$ \\
$0.025$ & $1.4655\times 10^{-13}$ & $1.2945\times 10^{-13}$ & $5.9952\times 10^{-15}$ & $4.7740\times 10^{-15}$ \\
 \hline
\end{tabular}

\begin{tabular}{ l || l l l l }
\multicolumn{5}{l}{(b) Standard Galerkin} \\
\hline
 \multicolumn{1}{c}{} &  \multicolumn{4}{l}{$r=1$, piecewise linear polynomials}   \\
 \hline
$\Delta x$ & $E_{\mathcal{M}}$ & $E_{\mathcal{I}}$ & $E_{\mathcal{H}}$ & $E_{\mathcal{E}}$ \\
 \hline
$0.1$ & $1.1546\times 10^{-14}$  & $7.5495\times 10^{-15}$ & $1.7655\times 10^{-5}$ & $2.2332\times 10^{-5}$ \\
$0.05$ & $6.2617\times 10^{-14}$ & $5.2847\times 10^{-14}$ & $5.5643\times 10^{-7}$ & $7.0383\times 10^{-7}$ \\
$0.025$ & $9.6367\times 10^{-14}$ & $8.0602\times 10^{-14}$ & $1.7425\times 10^{-8}$ & $2.2046\times 10^{-8}$ \\
 \hline
\end{tabular}

\begin{tabular}{ l || l l l l}
 \multicolumn{1}{c}{} &   \multicolumn{3}{l}{$r=2$, piecewise quadratic polynomials} \\
 \hline
$\Delta x$ & $E_{\mathcal{M}}$ & $E_{\mathcal{I}}$ & $E_{\mathcal{H}}$ & $E_{\mathcal{E}}$ \\
 \hline
$0.1$    & $3.7303\times 10^{-14}$ & $3.3973\times 10^{-14}$ & $1.7782\times 10^{-5}$ &  $2.2493\times 10^{-5}$ \\
$0.05$   & $1.2745\times 10^{-13}$ & $1.1036\times 10^{-13}$ & $5.5743\times 10^{-7}$ &  $7.0509\times 10^{-7}$\\
$0.025$ & $3.2907\times 10^{-13}$ & $2.8866\times 10^{-13}$ & $1.7437\times 10^{-8}$ & $2.2056\times 10^{-8}$ \\
 \hline
\end{tabular}

\begin{tabular}{ l || l l l l}
 \multicolumn{1}{c}{} &  \multicolumn{3}{l}{$r=3$, piecewise cubic polynomials} \\
 \hline
$\Delta x$ & $E_{\mathcal{M}}$ & $E_{\mathcal{I}}$ & $E_{\mathcal{H}}$ & $E_{\mathcal{E}}$ \\
 \hline
$0.1$ & $8.8107\times 10^{-13}$ & $7.2564\times 10^{-13}$ & $1.7782\times 10^{-5}$ & $2.2493\times 10^{-5}$ \\
$0.05$ & $1.6684\times 10^{-12}$ & $1.3574\times 10^{-12}$ & $5.5743\times 10^{-7}$ & $7.0510\times 10^{-7}$ \\
$0.025$ & $3.5998\times 10^{-12}$ & $2.9947\times 10^{-12}$ & $1.7439\times 10^{-8}$ & $2.2059\times 10^{-8}$ \\
 \hline
\end{tabular}
\end{table}

In Figures \ref{fig:solwave0}(c) and (d) we observe that the tails generated by the two
methods are almost identical in the cases $r=1$ and $r=3$, and they have magnitude less than $10^{-5}$ for the same values
of $\Delta x$ and $\Delta t$ as before. In the case of $r=2$ the tails were different and of twice the magnitude compared to the corresponding tails in the cases of $r=1$ and $r=3$. Since the cases with even values of $r$ appear to be special cases that have not much to offer compared to the cases with odd values of $r$ we do not show details of the results here and we focus only on the cases $r=1$ and $r=3$.

Apparently, the conservative Galerkin finite element method
achieves better resolution for the propagation of a solitary wave even for the specific values of $\Delta x$ and $\Delta t$ compared to the non-conservative method by $T=1000$, as it is expected.
In particular, the magnitude of the trailing tails grows with time if the non-conservative method is applied, as shown in
Figure \ref{fig:sw22}. This phenomenon is caused mainly by the dissipative hump generated
at one end of the solitary pulse and in front of the rest of the dispersive
tails, and can be observed when non-conservative methods are used \cite{AD2001,DLM2003}. In order to observe the dissipative hump we ran a similar experiment as before where we
considered the propagation of the same solitary wave (translated initially at $x_0=-120$) in the interval $[-150,150]$.  Figure \ref{fig:sw33} shows magnifications to the trailing dispersive tails and the dissipative humps before they start interacting with the solitary pulse.
Although the dispersive tails in the conservative and non-conservative methods
are very similar, the dissipative hump can be observed only in the non-conservative
method. The dissipative hump always remains attached to the solitary pulse and
grows in length with time, dissipating the mass and the energy of the solitary
wave to the rest of the domain. 
\begin{figure}[ht!]
  \centering
  \includegraphics[width=\columnwidth]{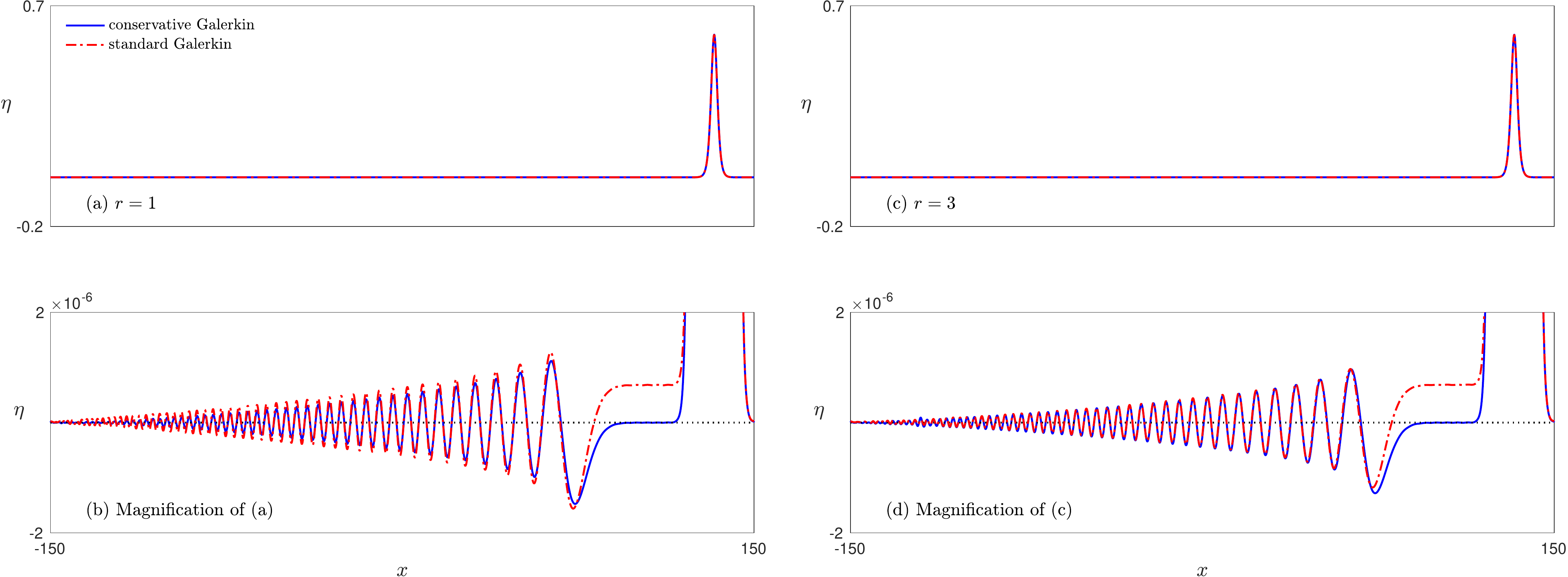}
  \caption{The propagation of a solitary wave with $c_s=\sqrt{1.6}$ and the formation of trailing tails in $[-150,150]$ at $T=200$, with  $\Delta x=\Delta t=0.1$.}
  \label{fig:sw33}
\end{figure}

It is worth mentioning that the accuracy of the derivative of the initial condition $\tilde{w}(x,0)$ is crucial for the accuracy in the propagation of the solitary waves. The magnitude of the tails indicates
analogous magnitude of errors in amplitude and shape. These errors are
$zero$ initially as we compare with the initial condition (translated by $c_s t$
units) and are increasing until they reach approximately a plateau where they
become oscillatory (at least in the conservative case). The phase error is always increasing due to the numerical error in the speed of the travelling wave. In Table \ref{tab:errorsind} we
present the mean values over the time interval $[80,100]$ of the amplitude,
phase and shape errors.

We observe that the errors are decreasing as we increase the number of elements
(decreasing $\Delta x$) and also as we increase the degree of the elements $r$.
In the case of piecewise linear elements the phase error remained of the same order
of magnitude when we chose different values of $\Delta x$ indicating that the
error in the phase speed is independent of the grid size. Finally, it is noted
that the errors are smaller with the conservative method compared to the
standard Galerkin method at least up to $T=1000$ with the exception of the case of quadratic Lagrange elements. The situation for $r=2$ is similar but the trailing tails are larger than the expected ones and thus the errors are also larger. This is perhaps again due to the choice of the initial condition.
On the other hand, the orders of magnitude of the errors are the same for the two
methods, indicating similarity in the propagation of travelling waves for small time scales.

In all the experiments conserning the propagation of a classical solitary wave we
considered the conservation of the invariants $\mathcal{M}$, $\mathcal{I}$,
$\mathcal{H}$ and $\mathcal{E}$. Table \ref{tab:errorsinvars} shows the errors
$E_\mathcal{M}$, $E_\mathcal{I}$, $E_\mathcal{H}$ and $E_\mathcal{E}$ defined
as $E_\mathcal{K}:=\max_{n}|\mathcal{K}(t^n)-\mathcal{K}(t^0)|$, where $\mathcal{K}$ is
one of  $\mathcal{M}$,  $\mathcal{I}$, $\mathcal{H}$, $\mathcal{E}$. We observe
that the conservative method conserves successfully the expected
conserved quantities. The standard Galerkin method obviously does not conserve the total energy
$\mathcal{E}$ although the errors in the computed energy are small. It is
interesting, that although the quantity $\mathcal{H}$ is conserved by the standard
Galerkin method but not by the conservative Galerkin method, the Runge--Kutta method does not
preserve it and thus it is not preserved by either of the two 
fully-discrete methods. The approximations obtained by the conservative
Galerkin method  are more accurate, while in the case of cubic Lagrange polynomials
the accuracy of this particular quantity is as if it were conserved by the numerical
method. As we shall see soon this was a coincidence.
A combination of the standard Galerkin method with an appropriate relaxation Runge--Kutta 
method could result in a fully-discrete scheme conserving the quantity
$\mathcal{H}$ when we consider periodic boundary conditions. Since the particular
quantity cannot be used in the case of other boundary conditions we refrain from
using such a combination.

One of the advantages of the conservative  method is the conservation of
the invariants for simulations over long time intervals. Letting the solitary wave
travel until $T=1000$ using the interval $[-20,20]$ we recorded the various errors and invariants. The errors
of the conservative  method are compared against the corresponding errors of
the standard Galerkin method in Figure \ref{fig:longt1} where the amplitude,
phase and shape errors are presented in the case of cubic Lagrange elements until
$T=1000$.  After an initial increase the amplitude and shape errors are stabilised
while the phase error increases linearly with time in the case of the conservative
Galerkin method in contrast with the standard Galerkin method, where the increase seems to be
at least of quadratic order. The errors are in general oscillatory, and in the
case of linear elements the errors are highly oscillatory and averaging leads to the same observations. A similar behaviour has been observed in the study of the
propagation of a solitary wave governed by the KdV equation using a conservative
discontinuous Galerkin method \cite{BCKX2013}. 

 \begin{figure}[ht!]
  \centering
  \includegraphics[width=\columnwidth]{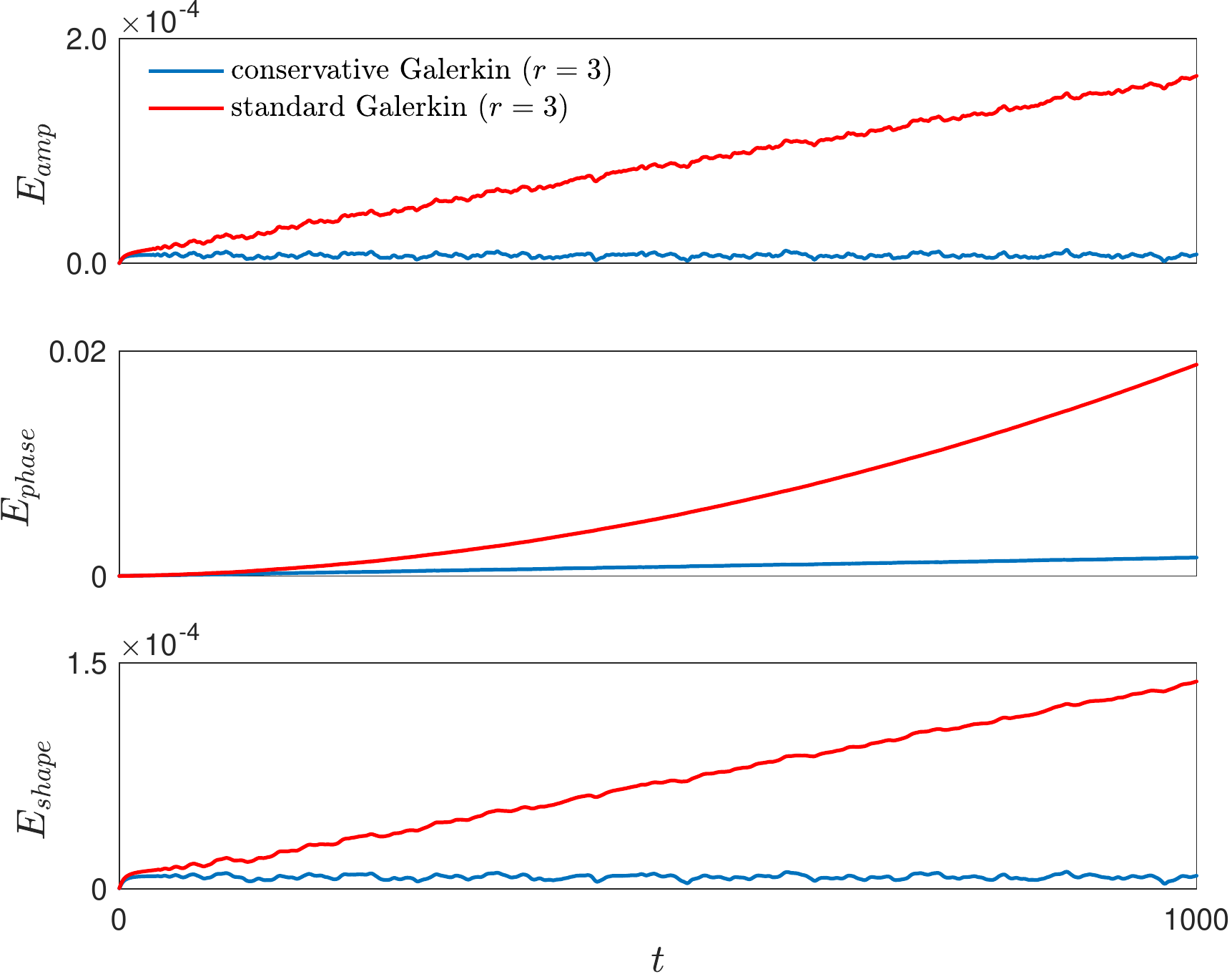}
  \caption{Amplitude, phase and shape errors for the propagation of a solitary wave with $c_s=\sqrt{1.6}$ up to $T=1000$, with $\Delta x=\Delta t= 0.1$.}
  \label{fig:longt1}
\end{figure}

\begin{table}[ht!]
\caption{Conservation of invariants $\mathcal{M}$, $\mathcal{I}$, $\mathcal{H}$, $\mathcal{E}$  in the propagation of a classical solitary wave with speed $c_s=\sqrt{1.6}$, ($\Delta x=\Delta t=0.1$, $T=1000$)}
\label{tab:longtimeerrs}
\begin{tabular}{l || l l l l}
 \hline
Method & $E_{\mathcal{M}}$ & $E_{\mathcal{I}}$ & $E_{\mathcal{H}}$ & $E_{\mathcal{E}}$ \\
 \hline
Conservative Galerkin ($r=1$) & $1.2879\times 10^{-14}$ & $8.6597\times 10^{-15}$ & $2.3305\times 10^{-9}$ & $2.9976\times 10^{-15}$ \\
Standard Galerkin ($r=1$) & $1.4655\times 10^{-14}$ & $8.6597\times 10^{-15}$ & $1.7631\times 10^{-4}$ & $2.2301\times 10^{-4}$ \\
Conservative Galerkin ($r=3$) & $2.7534\times 10^{-13}$ & $2.4225\times 10^{-13}$ & $3.0492\times 10^{-11}$ & $3.4417\times 10^{-15}$ \\
Standard Galerkin ($r=3$) & $8.4475\times 10^{-12}$ & $6.8885\times 10^{-12}$ & $1.7758\times 10^{-4}$ & $2.2462\times 10^{-4}$ \\
 \hline
\end{tabular}
\end{table}

The conservation of the quantities $\mathcal{M}, \mathcal{I}, \mathcal{H}$ and $\mathcal{E}$ is very similar across the various experiments. The corresponding errors are presented in Table \ref{tab:errorsinvars} with some differences shown in Table \ref{tab:longtimeerrs}. The conserved quantities remain almost the same with the exception of loosing 1 digit in some cases. Therefore, the conservative method can preserve the amplitude and shape of solitary waves during their propagation over long time intervals. Note that the values of $\gamma^n$ for $r=1$ satisfied $5.472\times 10^{-6}<\gamma^n-1<5.473\times 10^{-6}$ while for $r=3$ the respective interval was $2.145\times 10^{-8}<\gamma^n-1<2.146\times 10^{-8}$. Thanks to the conservation of the energy, it is expected that the specific conservative method leads to accurate numerical results in more general situations related to the propagation of solitary waves. Such an example is the interaction of two solitary waves travelling in the same direction. The particular experiment is described in the next section.

\section{Numerical experiments}\label{sec:exper}

In this section we repeat well-known numerical experiments for nonlinear and
dispersive waves, and we study the conservation properties of the new scheme.
In particular, we considered the following cases: (i) the overtaking collision of
two solitary waves and (ii) the reflection of a solitary wave by a vertical wall.
The first experiment is challenging as it involves integration in a long temporal interval
(thus conservation of energy should be important) and also during the interaction
of solitary waves it is known that various small-amplitude structures can
develop, such as dispersive tails and N-shaped wavelets \cite{ADM2010}.

\subsection{Overtaking collision of two solitary waves}

We first consider the interaction of two solitary waves travelling in the
same direction. The two solitary waves, generated using the Petviashvili method
as it is described in Appendix \ref{sec:appendix}, have phase speeds $c_s=1.6$
and $1.4$, respectively. It is known that various small-amplitude dispersive tails
emerge after the interaction of solitary waves, which then travel in both directions, \cite{ADM2010}. The resolution of the grid with $\Delta x=0.1$ is adequate to capture the complete picture of the
interaction with high resolution and without polluting the solution with small
amplitude numerical artefacts.

\begin{figure}[h!]
  \centering
  \includegraphics[width=\columnwidth]{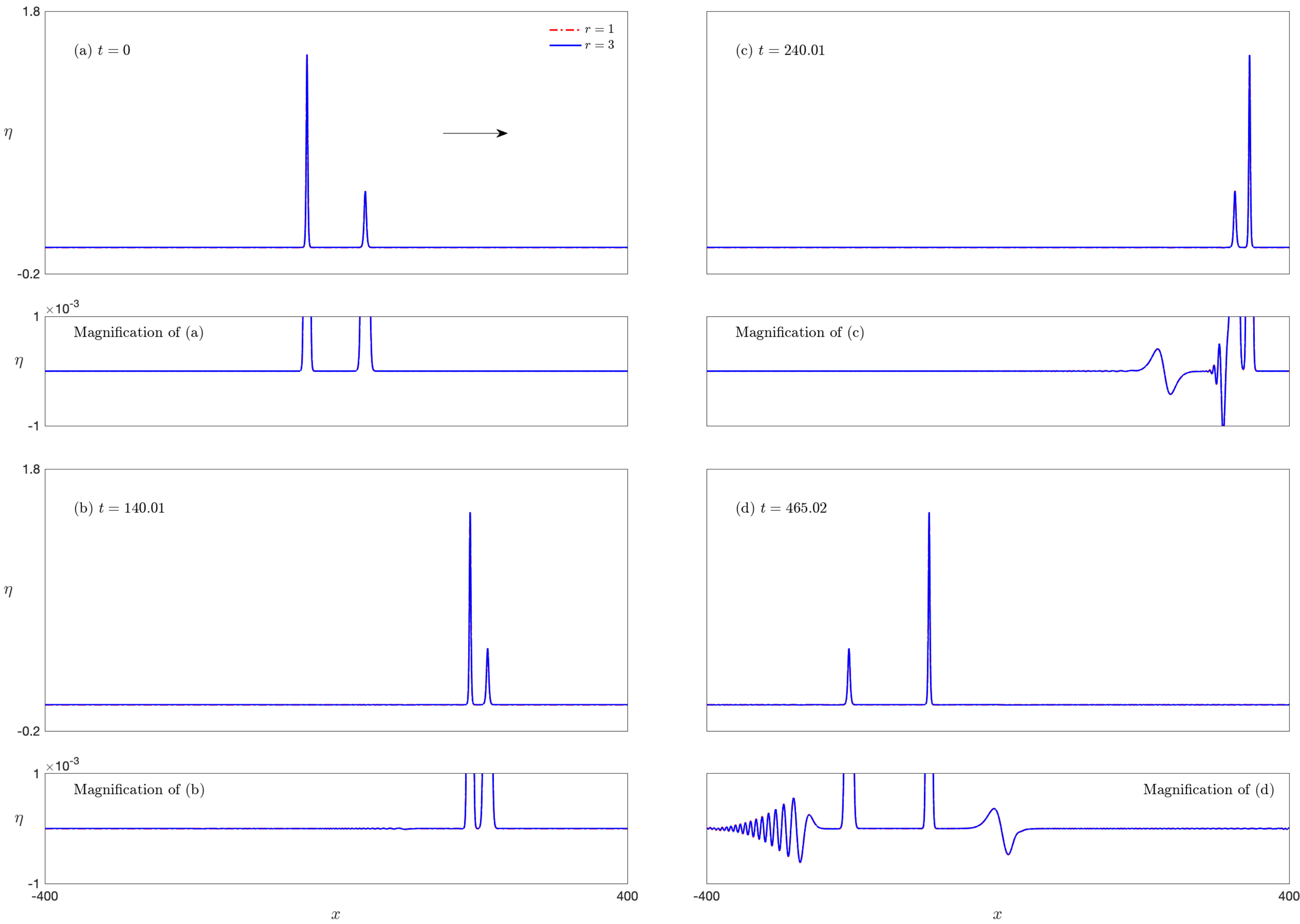}
  \caption{Overtaking collision of two solitary waves with speeds $c_s=1.6$ and $1.4$ and magnifications of trailing tails, subject to periodic boundary conditions, using the conservative Galerkin method.}
  \label{fig:ovcol1}
\end{figure}
Figure \ref{fig:ovcol1} presents snapshots before, during and after the
interaction of the two solitary waves at $t=0, 240.01, 140.01, 465.02$
(times rounded to two decimal digits) with Lagrange elements $\mathcal{P}^1$
and $\mathcal{P}^3$ ($\Delta x=\Delta t=0.1$).
We observe that the specific discretisation parameters were adequate to describe
all of the small-amplitude waves generated by the interaction of the two
travelling waves.
The interaction as approximated with the conservative Galerkin method with
linear elements appears to have acceptable resolution
and comparable to the case of cubic elements with the same values $\Delta x=\Delta t=0.1$.
It has been observed in \cite{ADM2010} that in the case of finite element spaces
with cubic splines the trailing tails are of negligible order of $10^{-10}$ for
$\Delta x=0.1$.

The values of the conserved quantities are shown in Table \ref{tab:invsover}.
The conserved quantities agree to almost all digits, except for the energy because the first derivative of the initial condition differs between $r=1$ and $r=3$. The errors in the conserved quantities are presented also
in Table \ref{tab:invsover}. It is easily observed that the errors of the conserved quantities in the case of $r=1$ have the same order as those with $r=3$.

\begin{table}[ht!]
\caption{Conserved invariants $\mathcal{M}$, $\mathcal{I}$, $\mathcal{H}$,
$\mathcal{E}$ and their errors, during the interaction of two solitary waves with $c_s=1.6$ and
$1.2$ respectively, ($\Delta x=\Delta t=0.1$, $T=600$). Periodic boundary conditions, conservative Galerkin method}
\label{tab:invsover}
\begin{tabular}{llllll}
\hline
$\mathcal{P}^r$ &  $\mathcal{M}$ & $\mathcal{I}$ & $\mathcal{H}$ & $\mathcal{E}$ \\
\hline
$r=1$ &  $5.587664655393$ & $4.776980337400$ & $3.711$ & $5.02174343002$  \\
$r=3$ &  $5.58766465539$ & $4.77698033740$ & $3.7115$ & $5.0217441007$ \\
\hline
$\mathcal{P}^r$ &  $E_{\mathcal{M}}$ & $E_{\mathcal{I}}$ & $E_{\mathcal{H}}$ & $E_{\mathcal{E}}$ \\
 \hline
$r=1$ & $7.7272\times 10^{-14}$ & $7.2831\times 10^{-14}$ & $2.0891\times 10^{-4}$ & $2.7445\times 10^{-12}$ \\
$r=3$ & $5.4623\times 10^{-13}$ & $4.3610\times 10^{-13}$ & $8.0327\times 10^{-5}$ & $2.3537\times 10^{-12}$ \\
\hline
\end{tabular}
\end{table}

In these experiments the values of $\gamma^n$ remained bounded in the interval $4.090\times 10^{-6} < \gamma^n-1<  4.205\times 10^{-5}$ with differences after the fourth decimal digit for both $r=1$ and $r=3$. It is important to note that the initial conditions are generated in $\mathcal{P}^3$ and then interpolated into the space $\mathcal{P}^1$. This observation is crucial because using different approximations to the initial conditions will lead to different errors.

\subsection{Solitary wave reflection by vertical wall}

Hitherto we have been concerned with the conservation properties of classical solitary waves in the
case of periodic boundary conditions. In this section we study the conservation
of mass and energy in the case of the reflection of a classical solitary wave
by a vertical wall. The quantity $\mathcal{I}$ is not preserved by the solution
of the BBM-BBM system in the case of reflective boundary conditions ($\eta_x=u=0$).
Because the reflection of a solitary wave by a vertical wall is equivalent to the
head-on collision of the same solitary wave with its symmetric image (propagating
in the opposite direction) about the wall, we expect that the conservation
properties will be very similar to the periodic case.

In this experiment we generate a right-travelling solitary wave with phase speed
$c_s=1.6$ in the interval $[-40,40]$. Because our numerical method is practically
a mixed Galerkin method, Neumann boundary conditions are essential and are enforced
explicitly. In both cases of $\mathcal{P}^r$ elements with $r=1$ and $r=3$ we use
$\Delta x=\Delta t=0.1$. Figure \ref{fig:reflx1} presents a graphical comparison
of the two methods for the reflection of the solitary wave. We observe that the
specific values we used for the discretisation parameters are adequate to
capture the reflection with high accuracy. The small artificial
tails generated due to the approximation of the initial conditions are negligible
compared to the tails generated because of the interaction of the travelling wave with
the vertical wall.
\begin{figure}[h!]
  \centering
  \includegraphics[width=\columnwidth]{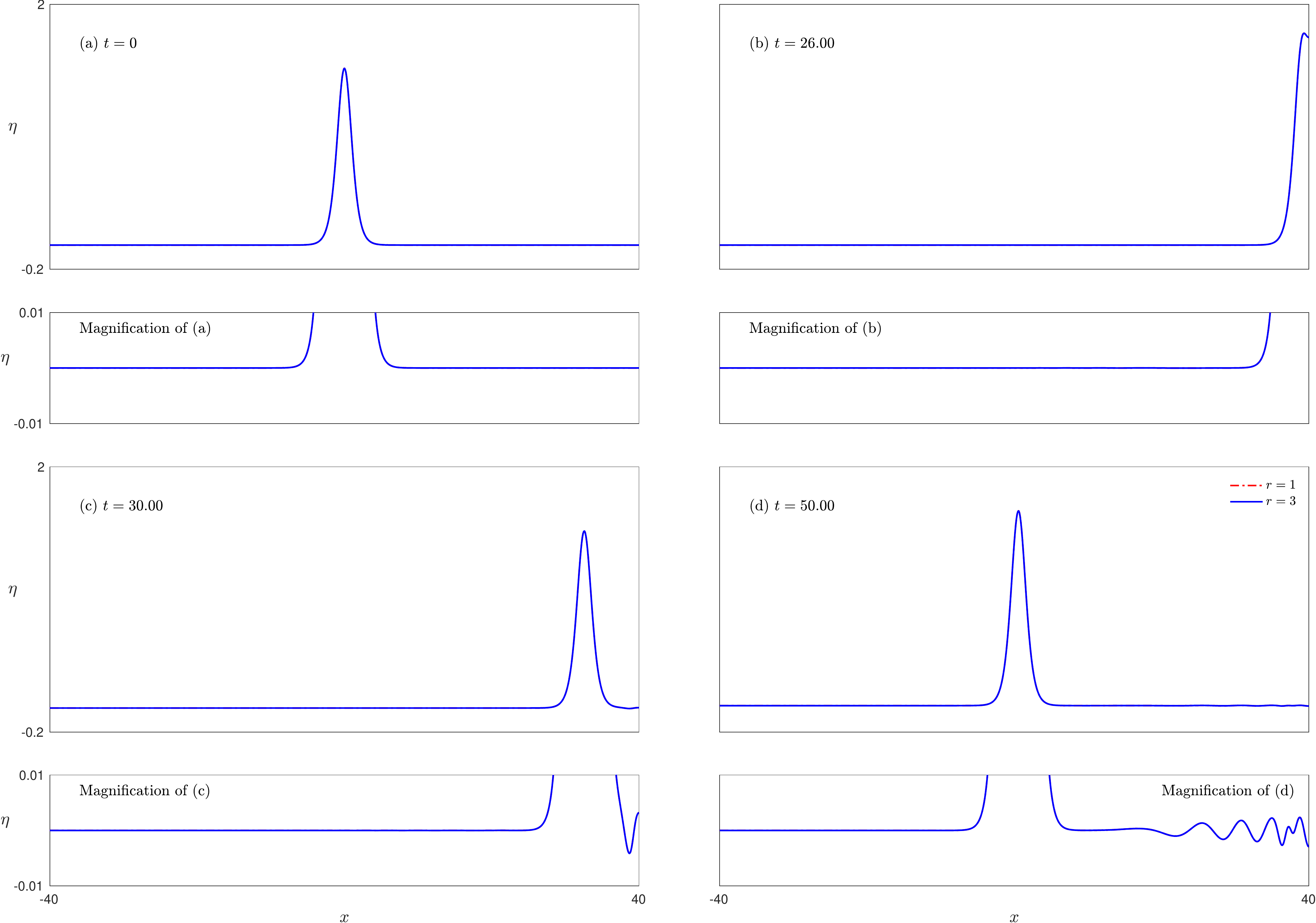}
  \caption{Reflection of a solitary wave with speed $c_s=1.6$ by a vertical wall and magnifications of trailing tails. Reflective boundary conditions}
  \label{fig:reflx1}
\end{figure}
\begin{table}[ht!]
\caption{Conserved invariants $\mathcal{M}$, $\mathcal{E}$ and their errors during the reflection of a solitary wave with speed $c_s=1.6$, ($\Delta x=\Delta t=0.1$, $T=50$). Reflective boundary conditions, conservative Galerkin method}
\label{tab:reflx1}
\begin{tabular}{lllllll}
\hline
$\mathcal{P}^r$ &   $\mathcal{M}$ & $\mathcal{E}$ & $E_{\mathcal{M}}$ & $E_{\mathcal{E}}$  \\
\hline
$r=1$ &  $3.8787933082344$ & $4.4967420062505$ & $8.8818\times 10^{-15}$ & $1.5987\times 10^{-14}$ \\
$r=3$ &  $3.8787933082344$ & $4.4967426642502$ & $3.8192\times 10^{-14}$ &  $1.5099\times 10^{-14}$ \\
\hline
\end{tabular}
\end{table}

The quantities $\mathcal{M}$ and $\mathcal{E}$ are presented in
Table \ref{tab:reflx1}. The quantities $\mathcal{I}$ and $\mathcal{H}$ are not
preserved even by the exact solution, and thus are not reported here.
The conserved quantities have similar accuracy as in the periodic case and the
reflective boundary conditions do not interfere with the conservation properties
of the method. The errors in these quantities are also presented in Table \ref{tab:reflx1}.

Note that for $r=1$ and $r=3$ the values of $\gamma^n$ where very similar and remained bounded in the same interval $6.447\times 10^{-6}<\gamma^n-1<6.346\times 10^{-5}$.

\section{Conclusions}\label{sec:conclu}

A new fully-discrete scheme for the regularised shallow water Boussinesq system of
equations was studied in the cases of periodic and reflective boundary conditions.
The new numerical method is based on a mixed Galerkin finite element semidiscretisation
in space and a relaxation Runge--Kutta method in time. Experimental studies of the
spatial convergence rates reveal an unusual even/odd dichotomy, where optimal rates
are observed only for piecewise polynomials of odd degree on spatially uniform grids. Due to the nature of the
mixed finite element method, superconvergence is observed for the approximation
in the $H^1$ norm. The new numerical method appears to have similar stability properties while achieving better
accuracy  compared to its non-conservative standard Galerkin counterpart.
The main advantage of the new method is the conservation of the total energy
functional and its potential for generalisation to multiple dimensions. Moreover, a study of the propagation of classical solitary waves shows that the numerical solitary waves can travel almost without changing shape over long time intervals by using the new conservative method. The dissipative hump that can be observed in the propagation of a solitary wave when nonconservative methods are used is not observed with the new conservative method. The new method performs very well even when we use low-order finite element methods with coarse grids. This makes
the conservative method desirable for long-time simulations and especially when good conservation properties are required by the
numerical scheme. The quality of the conservation of the total energy is not affected by the choice of the boundary conditions.

\appendix

\section{Petviashvili iteration}\label{sec:appendix}

Although Boussinesq systems of BBM-BBM type are known to possess classical
solitary waves \cite{BC1998,Chen1998}, there are no known closed-form solution formulas 
for such systems. For this reason computational methods are usually employed to generate solutions
numerically. In this paper we consider the Petviashvili method \cite{Petv1976}.
The Petviashvili method is a modified fixed point algorithm for solving nonlinear
equations, originally derived for computing travelling wave solutions. Assuming
that the solution of (\ref{eq:ibvp1}) is a travelling wave of speed $c_s>1$ which
satisfies the {\em ansatz}
\begin{align}\label{eq:ansatz}
\eta(x,t)=\eta(\xi),\qquad u(x,t)=u(\xi),
\end{align}
with $\xi=x-c_s t-x_0$, $x_0\in\mathbb{R}$, substitution in (\ref{eq:ibvp1}) implies that
\begin{align}\label{eq:ode1}
\begin{aligned}
&-c_s\left(\eta-\frac{1}{6}\eta_{\xi\xi}\right)+(1+\eta)u=0,\\
&-c_s\left(\eta-\frac{1}{6}u_{\xi\xi}\right)+\eta+\frac{1}{2}u^2=0.
\end{aligned}
\end{align}
System (\ref{eq:ode1}) is then written in the form
\begin{align}\label{eq:petv1}
\mathcal{L}\bw=\mathcal{N}(\bw),
\end{align}
where $\bw=(\eta,u)^\mathrm{T}$,
\begin{align}\label{eq:petv2}
\mathcal{L}=\begin{pmatrix}
c_s\left(1-\frac{1}{6}\partial_{\xi}^2\right) & 1\\
1 & c_s\left(1-\frac{1}{6}\partial_{\xi}^2\right)
\end{pmatrix}\quad \text{and}\quad
\mathcal{N}(\bw)=\begin{pmatrix}
\eta u\\
\frac{1}{2}u^2
\end{pmatrix}.
\end{align}
The Galerkin finite element method for solving (\ref{eq:petv1}), (\ref{eq:petv2}) is the following. Let $\bS_h$ be one of the spaces $\mathcal{P}^r_p\times \mathcal{P}^r_p$ or $\mathcal{P}^r\times \mathcal{P}^r_0$; seek an approximation $\bw_h=(\eta_h, u_h)^\mathrm{T}\in S_h$ such that
\begin{align}\label{eq:petv3}
\mathcal{L}_h(\bw_h,\bchi)=(\mathcal{N}(\bw_h),\bchi), \quad \text{for all}\quad \bchi\in \bS_h,
\end{align}
where
\begin{align}\label{eq:petv4}
\mathcal{L}_h (\bw,\bchi):=c_s(\eta,\phi)+\frac{1}{6}c_s(\eta_\xi,\phi_\xi)-(w,\phi)+c_s(w,\chi)+\frac{1}{6}c_s(w_\xi,\chi_\xi)-(\eta,\chi),
\end{align}
for all $\bw=(\eta,u)^\mathrm{T}\in \bS_h$ and $\bchi\in \bS_h$. Given an initial guess $\bw_h^0$ for $\bw_h$, the Petviashvili iteration for solving the nonlinear system of equations (\ref{eq:petv3}), (\ref{eq:petv4}) is then defined as
\begin{align}
\mathcal{L}_h(\bw_h^{n+1},\bchi)=M_n^\gamma (\mathcal{N}(\bw_h^n),\bchi),\quad \text{for all} \quad \bchi\in S_h,\quad n=0,1,\ldots,
\end{align}
where $M_n$ is defined as
$$M_n:=\frac{\mathcal{L}_h(\bw_h^n,\bw_h^n)}{(\mathcal{N}(\bw_h^n),\bw_h^n)}.$$
As an initial guess $\bw_h^0$ for the approximation of the solution $\bw_h$ we consider the $L^2$-projection of a solitary wave of the form $(\eta^0(\xi),u^0(\xi))^\mathrm{T}=(A\,\sech^2(\lambda\xi),\ c_s\eta^0(\xi)/(1+\eta^0(\xi))$, with $c_s:=\sqrt{1+A}$ and $\lambda:=\sqrt{3A/4}$. The exponent $\gamma$ can be any number in the interval $[1,3]$. As a stopping criterion for the Petviashvili iteration we took
$$R_n:=\frac{|\mathcal{L}_h(\bw_h^n,\bw_h^n)-(\mathcal{N}(\bw_h^n),\bw_h^n)|}{\|\bw_h^n\|_2}<\delta,$$ 
corresponding to the normalised residual $R_n$ in the $n$-th iteration 
falling below an appropriate prescribed tolerance $\delta$.
In all the experiments considered in this paper we took $\gamma=2$ and $\delta=10^{-10}$.
\begin{figure}[h!]
  \centering
  \includegraphics[width=0.8\columnwidth]{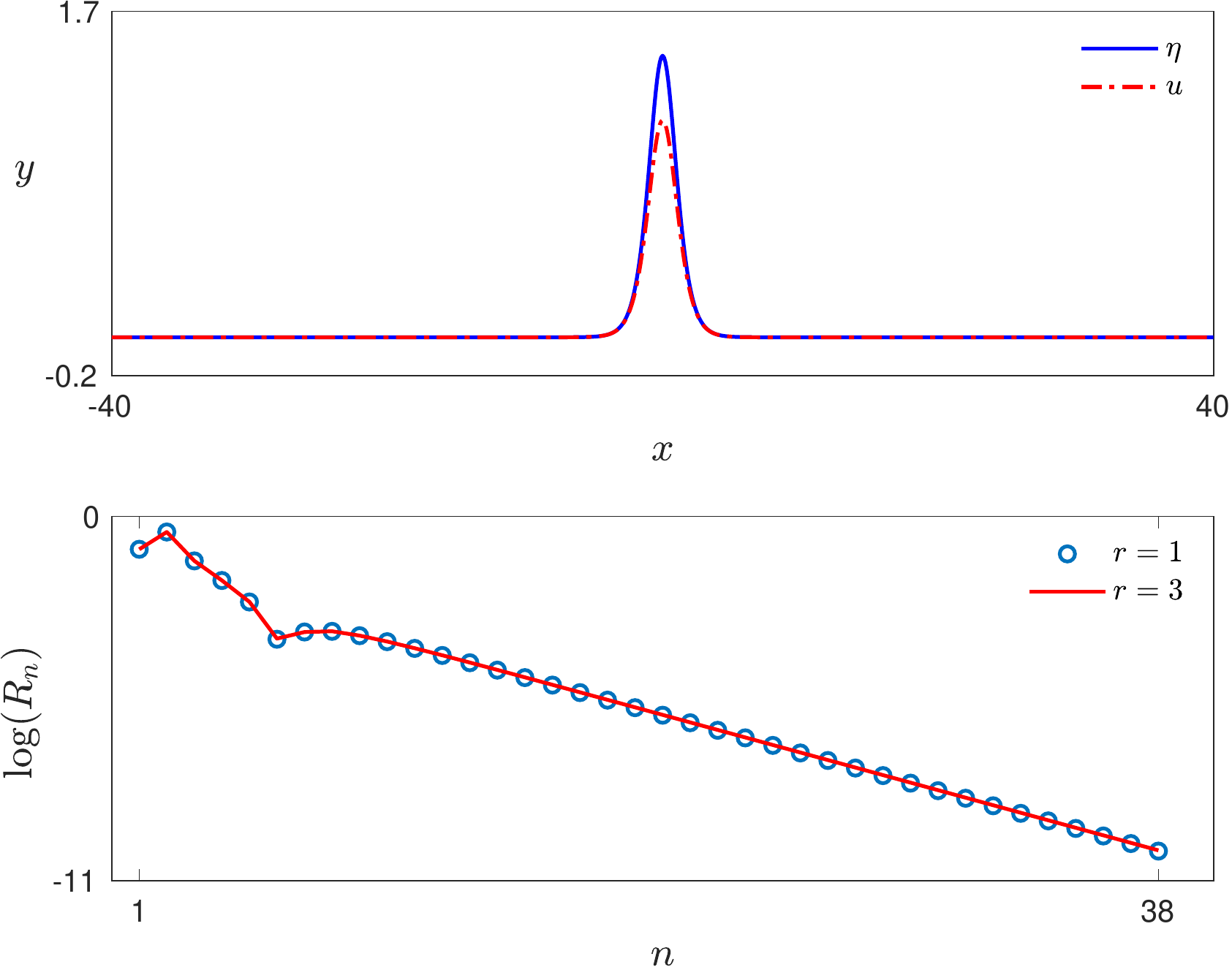}
  \caption{Generation of a solitary wave with speed $c_s=1.6$ using the Petviashvili method. Top panel: solution $(\eta, u)$; bottom panel: logarithm of the residual $R_n$}
  \label{fig:petvia1}
\end{figure}

We demonstrate the convergence of the Petviashvili method by generating a
solitary wave with phase speed $c_s=1.6$ in the interval $[-40,40]$ with both
periodic and reflective boundary conditions. Figure \ref{fig:petvia1} depicts
the solution generated and the logarithm of the residual \textit{versus} iteration number.
The results for cubic and linear elements coincide and also the different
boundary conditions had no effect on the convergence to the solution. In all
cases considered $38$ iterations were required to ensure that $|R_n|\le 10^{-10}$. On the other hand, although the convergence appears to be very similar for cubic and linear elements, including the residuals, the resolution of the numerical solution with cubic elements is better since the derivatives have been approximated with higher accuracy.

\bibliographystyle{plain}
\bibliography{biblio}
\bigskip

\end{document}